\expandafter\def\csname ver@fixltx2e.sty\endcsname{}
 \documentclass[journal,9pt]{IEEEtran}

\usepackage{amsthm}
\usepackage{amsfonts}
\usepackage{amssymb,bm}
\usepackage{mathrsfs}
\usepackage{mathtools}
\usepackage{amsmath}

\usepackage[switch]{lineno}
\usepackage[named]{algo}
\usepackage[noend]{algpseudocode}
\usepackage{algorithm}

\usepackage{cite}
\usepackage{balance}

\let\labelindent\relax
\usepackage[inline]{enumitem}

\usepackage{verbatim}

\usepackage{dblfloatfix}
\usepackage{url}
\usepackage{graphicx}
\usepackage{textcomp}

\usepackage{accents}

\newlength{\dhatheight}
\newcommand{\doublehat}[1]{%
    \settoheight{\dhatheight}{\ensuremath{\hat{#1}}}%
    \addtolength{\dhatheight}{-0.25ex}%
    \hat{\vphantom{\rule{1pt}{\dhatheight}}%
    \smash{\hat{#1}}}}

\usepackage{array}
\newcolumntype{P}[1]{>{\centering\arraybackslash}p{#1}}

\usepackage{xcolor}

\newtheorem{theorem}{Theorem}

\newtheorem{lemma}{Lemma}

\newtheorem{remark}{Remark}

\def\BibTeX{{\rm B\kern-.05em{\sc i\kern-.025em b}\kern-.08em
    T\kern-.1667em\lower.7ex\hbox{E}\kern-.125emX}}

\begin{document}

\title{Inverse Extended Kalman Filter --- Part II: Highly Non-Linear and Uncertain Systems}

\author{Himali Singh, Arpan Chattopadhyay$^\ast$ and Kumar Vijay Mishra$^\ast$
\thanks{$^\ast$A. C. and K. V. M. have made equal contributions.}
\thanks{H. S. and A. C. are with the Electrical Engineering Department, Indian Institute of Technology (IIT) Delhi, India. A. C. is also associated with the Bharti School of Telecommunication Technology and Management, IIT Delhi. Email: \{eez208426, arpanc\}@ee.iitd.ac.in.} 
\thanks{K. V. M. is with the United States DEVCOM Army Research Laboratory, Adelphi, MD 20783 USA. E-mail: kvm@ieee.org.}
\thanks{A. C. acknowledges support via the professional development fund and professional development  allowance from IIT Delhi, grant no. GP/2021/ISSC/022 from I-Hub Foundation for Cobotics and grant no. CRG/2022/003707 from Science and Engineering Research Board (SERB), India. H. S. acknowledges support via Prime Minister Research Fellowship. K. V. M. acknowledges support from the National Academies of Sciences, Engineering, and Medicine via Army Research Laboratory Harry Diamond Distinguished Fellowship.}
}

\maketitle

\begin{abstract}
Counter-adversarial system design problems have lately motivated the development of inverse Bayesian filters. For example, \textit{inverse Kalman filter} (I-KF) has been recently formulated to estimate the adversary's Kalman-filter-tracked estimates and hence, predict the adversary's future steps. The purpose of this paper and the companion paper (Part I) is to address the inverse filtering problem in non-linear systems by proposing an \textit{inverse extended Kalman filter} (I-EKF). The companion paper proposed the theory of I-EKF (with and without unknown inputs) and I-KF (with unknown inputs). In this paper, we develop this theory for highly non-linear models, which employ second-order, Gaussian sum, and dithered forward EKFs. In particular, we derive theoretical stability guarantees for the inverse second-order EKF using the bounded non-linearity approach. To address the limitation of the standard I-EKFs that the system model and forward filter are perfectly known to the defender, we propose reproducing kernel Hilbert space-based EKF to learn the unknown system dynamics based on its observations, which can be employed as an inverse filter to infer the adversary's estimate. Numerical experiments demonstrate the state estimation performance of the proposed filters using recursive Cram\'{e}r-Rao lower bound as a benchmark.
\end{abstract}


\section{Introduction}\label{sec:introduction}
In the companion paper (Part I) \cite{singh2022inverse_part1}, we introduced \textit{inverse cognition} \cite{krishnamurthy2019how} in a non-linear setting and developed inverse extended Kalman filters (I-EKFs). Recall that 
the inverse cognition problem involves two agents: `defender' (e.g., an intelligent target) and an `adversary' (e.g., a radar) equipped with a (forward) Bayesian tracker. The forward filter provides a posterior distribution of the underlying state to the adversary based on its noisy observations. The cognitive adversary further adapts its actions based on this estimate. For instance, a cognitive radar adapts its waveform to enhance target detection and tracking \cite{bell2015cognitive,sharaga2015optimal}. The defender observes these adversary's actions to predict its future actions in a Bayesian sense and hence, requires an estimate of the adversary's estimate. The inverse cognition problem is motivated by the design of counter-adversarial systems, e.g., an intelligent target observing its adversarial radar's actions to guard against the latter\cite{krishnamurthy2019how}. Other applications include interactive learning\cite{krishnamurthy2019how}, fault diagnosis, cyber-physical security\cite{mattila2017inverse}, and inverse reinforcement learning\cite{ng2000algorithms}.

An inverse Bayesian filter provides an estimate of the forward Bayesian filter's posterior distribution given its noisy measurements. In particular, \cite{mattila2020inverse} and \cite{krishnamurthy2019how} developed, respectively, the inverse hidden Markov model (I-HMM) filter and inverse Kalman filter (I-KF) 
for linear Gaussian state-space models. The tracked estimates from these inverse filters were then used to design smart interference and force the adversarial radar to change its waveform in \cite{krishnamurthy2021adversarial,kang2023}. In practice, many engineering processes are non-linear and employ extended Kalman filter (EKF) based on local linearization at the state estimates \cite{haykin2004kalman,simon2006optimal}. 

In the companion paper (Part I) \cite{singh2022inverse_part1}, we developed inverse EKF (I-EKF) for a non-linear inverse cognition problem. 
However, the standard EKF has its limitations. Unlike KF, the filter and (filter) gain equations of EKF are not decoupled and hence offline computations of these EKF quantities are not possible. The EKF performs poorly when the dynamic system is significantly non-linear. It is also very sensitive to initialization and may even completely fail \cite{tenney1977tracking}. Recently, KF's initialization has been studied rigorously in \cite{zhao2020trial}, while \cite{salvoldi2018process} deals with process noise covariance design. However, they do not consider the non-linear EKF case. These EKF drawbacks have led to the development of several advanced variants, each aiming toward improving either the estimation accuracy (higher-order EKFs \cite{jazwinski2007stochastic,wang2019second}, Gaussian-sum EKF \cite{tam1977gaussian}), convergence (iterated EKF \cite{wishner1969comparison}), stability (dithered EKF \cite{weiss1980improved}) or practical feasibility (hybrid EKF\cite{simon2006optimal}). The literature also reports grid-based EKF \cite{bucy1971digital} that is a pre-cursor of particle filter. A comparison of non-linear KFs is available in \cite{schwartz1968computational,netto1978optimal}. These filters reduce the linearization errors that are inherent to the EKF and may provide an improved estimation at the cost of higher complexity and computations. In particular, this paper focuses on developing the I-EKF theory for highly nonlinear systems that may employ some of the aforementioned EKF variants as the forward filter.

Further, the inverse stochastic filters developed in \cite{mattila2020inverse,krishnamurthy2019how} and the companion paper (Part I) \cite{singh2022inverse_part1} assume perfect system model information on the part of both adversary and defender. However, in practice, the defender 
may not have information about the adversary's forward filter and 
the strategy employed by the adversary to adapt its actions. Similarly, the adversary may also lack complete system information, including the noise distributions. In such 
situations, the performance and applicability of inverse filters is limited. The uncertainty in system parameters or noise statistics may result in large estimation errors or even filter divergence \cite{vila2021robust,ge2016performance}. Prior works \cite{vila2021robust,huang2017novel,zhu2021adaptive,yi2021robust,mohamed2011robust} proposed various adaptive filters to reduce the sensitivity of KF to system uncertainties. However, the general non-linear filtering problem is not considered. 

In this paper, we also consider the case of unknown forward filter and system model. Among prior works, for a forward non-linear filtering problem with unknown system model, \cite{liu2009extended} developed extended kernel recursive least squares (Ex-KRLS) method by transforming the input space to a kernel-based feature space. The KRLS\cite{engel2004kernel} was coupled with EKF to learn the unknown non-linear measurement model in \cite{zhu2011extended,singh2022rkhs}, but with a known linear state-transition. In \cite{zhu2013learning}, the KF algorithm was reformulated for non-linear state transition functions in reproducing kernel Hilbert space (RKHS) using a conditional embedding operator, but linear observations were assumed. Particle filter based approaches for state estimation and system identification in non-linear system setting were considered in \cite{tobar2015unsupervised,berntorp2021online}. For linear Gaussian state-space models, an iterative expectation maximization (EM)-based parameter learning was proposed in \cite{wen2012data}. We adopt this iterative EM for non-linear parameter learning in our RKHS-based EKF.
Our main contributions in this paper (Part II) are:\\
\textbf{1) Inverses of EKF variants.} We develop inverse second-order EKF (I-SOEKF) and provide sufficient conditions for its stability. In the process, we obtain stability conditions for forward SOEKF as well and derive the one-step prediction formulation of SOEKF. The presence of second-order terms poses additional challenges in deriving the theoretical stability guarantees for both forward and inverse SOEKF. In addition, we obtain the inverse of Gaussian-sum EKF (GS-EKF) which comprises of several individual EKFs \cite{tam1977gaussian,alspach1972nonlinear}. Here, the \textit{a posteriori} density function of the state given the observations is approximated by a sum of Gaussian density functions each of which has its own separate EKF. GS-EKF and hence, its inverse, not only provides enhanced estimation accuracy but is also a popular method to handle non-Gaussianity in system models. In situations where the estimation error is small, the \textit{a posteriori} density is approximated adequately by one Gaussian density, and the Gaussian sum filter reduces to the EKF; the same is true for their inverses. Finally, we also examine the inverse of dithered EKF (DEKF) \cite{weiss1980improved}, which has been shown to perform better when non-linearities are cone-bounded. This filter introduces dither signals prior to the non-linearities to tighten the cone-bounds and hence, improve stability properties. We validate the estimation errors of all inverse EKFs through extensive numerical experiments with recursive Cram\'{e}r-Rao lower bound (RCRLB) \cite{tichavsky1998posterior} as the performance metric.\\
\textbf{2) RKHS-EKF.} We consider the case when both the state-transition and observation models are unknown to the agent employing a stochastic filter in a non-linear system. In this context, we propose RKHS-based EKF which may be employed by both adversary and defender, respectively, to estimate defender's state and adversary's state estimate. In particular, the unknown non-linear functions are represented using kernel function approximation. The EKF provides required state estimates while an online approximate EM is used to learn the system parameters. The presence of non-linear transformations in expectation computation makes the filter derivation non-trivial. \\
\textbf{3) Mismatched forward and inverse filters.} The inverse filter theory formulated in the companion paper (Part I) \cite{singh2022inverse_part1} assumed a known adversary's forward filter. If the defender is not certain about the adversary's forward filter, our proposed RKHS-EKF is helpful in estimating the adversary's state estimate. However, when the defender employs an inverse filter assuming a forward filter that is not the same as the adversary's true forward filter, our numerical experiments show that the defender is still able to estimate the adversary's state estimate with reasonable accuracy. Further, a sophisticated inverse filter may even provide better performance at the cost of computational efforts.

The rest of the paper is organized as follows. In the next section, we describe the system model for the inverse cognition problem. We then discuss the inverses of some EKF variants in Section~\ref{sec:extensions}. The stability conditions for forward and inverse SOEKF are provided in Section~\ref{sec:stability}. In Section~\ref{sec:RKHS-EKF}, we relax the perfect system information assumption and derive RKHS-EKF. Finally, in Section~\ref{sec:simulations}, we demonstrate the performance of the proposed inverse filters considering various numerical experiments before concluding in Section~\ref{sec:summary}.

Throughout the paper, we reserve boldface lowercase and uppercase letters for vectors (column vectors) and matrices, respectively. The notation used to denote the $i$-th component of vector $\mathbf{a}$ and the $(i,j)$-th component of matrix $\mathbf{A}$ are $\left[\mathbf{a}\right]_{i}$ and $\left[\mathbf{A}\right]_{i,j}$, respectively, while $[\mathbf{A}]_{(i_{1}:i_{2},j_{1}:j_{2})}$ represents the sub-matrix of $\mathbf{A}$ consisting of rows $i_{1}$ to $i_{2}$ and columns $j_{1}$ to $j_{2}$. We use $(\cdot)^T$ for transpose operation while $||\cdot||_{2}$ represents the $l_2$ norm for a vector. For matrix $\mathbf{A}$, the trace, determinant, rank, spectral norm and maximum row sum norm are denoted by $\textrm{Tr}(\mathbf{A})$, $|\mathbf{A}|$, $\textrm{rank}(\mathbf{A})$, $||\mathbf{A}||$, and $\|\mathbf{A}\|_{\infty}$, respectively. The inequality $\mathbf{A}\preceq\mathbf{B}$  for matrices $\mathbf{A}$ and $\mathbf{B}$  means that $\mathbf{B}-\mathbf{A}$ is a positive semidefinite (p.s.d.) matrix. The $\mathbb{R}^{m \times n}$ Jacobian matrix of a function $f:\mathbb{R}^{n}\rightarrow\mathbb{R}^{m}$ is denoted by $\nabla f$. Similarly, $\nabla f$ and $\nabla^{2}f$ denote the gradient vector ($\mathbb{R}^{n\times 1}$) and Hessian matrix ($\mathbb{R}^{n\times n}$), respectively, for a function $f:\mathbb{R}^{n}\rightarrow\mathbb{R}$. $\mathbf{I}_{n}$ denotes a $n\times n$ identity matrix while $\mathbf{0}_{n\times m}$ denotes a $n\times m$ all zero matrix. The notation $\lbrace a_{i}\rbrace_{i_{1}\leq i\leq i_{2}}$ denotes a set of elements indexed by integer $i$. 
The Gaussian random variable is represented as $\mathbf{x} \sim \mathcal{N}(\boldsymbol{\mu},\mathbf{Q})$ with mean  $\boldsymbol{\mu}$ and covariance matrix $\mathbf{Q}$ while $x \sim \mathcal{U}[u_l,u_u]$ means a uniform distribution over $[u_l,u_u]$. The covariance of random variable $\mathbf{x}$ is denoted by $\textrm{Cov}(\mathbf{x})$ while $\textrm{Cov}(\mathbf{x},\mathbf{y})$ denotes its cross-covariance with random variable $\mathbf{y}$.

\section{Desiderata for Inverse Cognition}
\label{sec:background}
Throughout the paper, we focus on discrete-time models. Recall from the companion paper (Part I) \cite{singh2022inverse_part1} that the defender's discrete-time stochastic state process $\{\mathbf{x}_k\}_{k \geq 0}$, where $\mathbf{x}_k \in \mathbb{R}^{n \times 1}$ is the state at the $k$-th time instant, evolves as
\par\noindent\small
\begin{align}
\mathbf{x}_{k+1}=f(\mathbf{x}_{k})+\mathbf{w}_{k},\label{eqn: state transition x}
\end{align}
\normalsize
where $\mathbf{w}_{k}\sim\mathcal{N}(\mathbf{0}_{n\times 1},\mathbf{Q})$ is the process noise with covariance matrix $\mathbf{Q} \in \mathbb{R}^{n\times n}$. The defender knows its own state perfectly, while the adversary observes the defender's state in noise as $\mathbf{y}_k \in \mathbb{R}^{p \times 1}$ at time $k$ as
\par\noindent\small
\begin{align}
\mathbf{y}_{k}=h(\mathbf{x}_{k})+\mathbf{v}_{k},\label{eqn: observation y}
\end{align}
\normalsize
where $\mathbf{v}_{k}\sim\mathcal{N}(\mathbf{0}_{p\times 1},\mathbf{R})$ is the adversary's measurement noise with covariance matrix $\mathbf{R} \in \mathbb{R}^{p\times p}$.

The adversary's (forward) stochastic filter computes the estimate $\hat{\mathbf{x}}_k$ of the defender's state $\mathbf{x}_{k}$ using $\{\mathbf{y}_j\}_{1 \leq j \leq k}$. The adversary then takes an action $g(\hat{\mathbf{x}}_{k}) \in \mathbb{R}^{n_a \times 1}$ using this estimate, whose noisy observation made by the defender is
\par\noindent\small
\begin{align}
\mathbf{a}_{k}&=g(\hat{\mathbf{x}}_{k})+\bm{\epsilon}_{k},\label{eqn: observation a}
\end{align}
\normalsize
where $\bm{\epsilon}_{k}\sim \mathcal{N}(\mathbf{0}_{n_{a}\times 1},\bm{\Sigma_{\epsilon}})$ is the defender's measurement noise with covariance matrix $\bm{\Sigma}_{\epsilon} \in \mathbb{R}^{n_{a}\times n_{a}}$. Finally, in the inverse stochastic filter, the defender computes the estimate $\doublehat{\mathbf{x}}_k \in \mathbb{R}^{n \times 1}$ of the estimate $\hat{\mathbf{x}}_k $ using $\{\mathbf{a}_j, \mathbf{x}_j\}_{1 \leq j \leq k}$, with $\overline{\bm{\Sigma}}_{k}$ as the associated error covariance matrix of $\doublehat{\mathbf{x}}_{k}$. The noise processes $\{\mathbf{w}_{k}\}_{k \geq 0}$, $\{\mathbf{v}_{k}\}_{k \geq 1}$ and $\{\bm{\epsilon}_{k}\}_{k \geq 1}$ are mutually independent and i.i.d. across time. Note that the adversary and defender are entirely different agents employing independent sensors to observe each other. Furthermore, in the inverse filtering problem, the adversary is unaware that the defender is observing the former. Also, the process noise $\mathbf{w}_{k}$ represents the modeling uncertainties in state transition \eqref{eqn: state transition x}. Hence, these noises are mutually independent. We further assume that the adversary does not seek to hide its actions from the defender, which is a more challenging problem recently addressed in \cite{lourencco2020protect,lourencco2021hidden} and not the focus of this paper. Further, other recent works \cite{pattanayak2022inverse,pattanayak2022meta} have considered the case when the adversary is also aware of the defender's inverse cognitive nature.

If the adversary knows this discrete-time system model perfectly, standard EKF is the simplest and most widely used forward filter that can be employed by the adversary. With EKF as the known adversary's forward filter, we developed I-EKF in the companion paper (Part I) \cite{singh2022inverse_part1}. We briefly review the I-EKF recursions here for further discussions in subsequent sections.

Based on the forward EKF recursions and observation \eqref{eqn: observation y}, a time-varying state transition for the I-EKF is obtained as
\par\noindent\small
\begin{align}
    &\hat{\mathbf{x}}_{k+1}=\widetilde{f}_{k}(\hat{\mathbf{x}}_{k},\mathbf{x}_{k+1},\mathbf{v}_{k+1})\nonumber\\
    &=f(\hat{\mathbf{x}}_{k})-\mathbf{K}_{k+1}h(f(\hat{\mathbf{x}}_{k}))+\mathbf{K}_{k+1}h(\mathbf{x}_{k+1})+\mathbf{K}_{k+1}\mathbf{v}_{k+1},\label{eqn:IEKF state transition}
\end{align}
\normalsize
with the true state $\mathbf{x}_{k+1}$ and $\mathbf{v}_{k+1}$, respectively, being a known exogenous input and process noise term for the defender's I-EKF. Here, $\mathbf{K}_{k+1}$ is the forward EKF's gain matrix, which the I-EKF treats as a time-varying parameter of state transition \eqref{eqn:IEKF state transition}, approximating it using its own estimates as described in the companion paper (Part I) \cite{singh2022inverse_part1}. Finally, using state transition \eqref{eqn:IEKF state transition} and observation \eqref{eqn: observation a}, I-EKF follows from the EKF recursions as
\par\noindent\small
\begin{align}
    &\textit{Prediction:}\;\doublehat{\mathbf{x}}_{k+1|k}=\widetilde{f}_{k}(\doublehat{\mathbf{x}}_{k},\mathbf{x}_{k+1},\mathbf{0}_{p\times 1}),\nonumber\\
    &\overline{\bm{\Sigma}}_{k+1|k}=\widetilde{\mathbf{F}}^{x}_{k}\overline{\bm{\Sigma}}_{k}(\widetilde{\mathbf{F}}^{x}_{k})^{T}+\overline{\mathbf{Q}}_{k},\nonumber\\
    &\textit{Update:}\;\overline{\mathbf{S}}_{k+1}=\mathbf{G}_{k+1}\overline{\bm{\Sigma}}_{k+1|k}\mathbf{G}_{k+1}^{T}+\bm{\Sigma}_{\epsilon},\;\;\hat{\mathbf{a}}_{k+1|k}=g(\doublehat{\mathbf{x}}_{k+1|k}),\nonumber\\
    &\doublehat{\mathbf{x}}_{k+1}=\doublehat{\mathbf{x}}_{k+1|k}+\overline{\bm{\Sigma}}_{k+1|k}\mathbf{G}_{k+1}^{T}\overline{\mathbf{S}}_{k+1}^{-1}(\mathbf{a}_{k+1}-\hat{\mathbf{a}}_{k+1|k}),\label{eqn:IEKF state update}\\
    &\overline{\bm{\Sigma}}_{k+1}=\overline{\bm{\Sigma}}_{k+1|k}-\overline{\bm{\Sigma}}_{k+1|k}\mathbf{G}_{k+1}^{T}\overline{\mathbf{S}}_{k+1}^{-1}\mathbf{G}_{k+1}\overline{\bm{\Sigma}}_{k+1|k},\label{eqn:IEKF covariance update}
\end{align}
\normalsize
where $\widetilde{\mathbf{F}}^{x}_{k}\doteq\nabla_{\mathbf{x}}\widetilde{f}_{k}(\mathbf{x},\mathbf{x}_{k+1},\mathbf{0}_{p\times 1})|_{\mathbf{x}=\doublehat{\mathbf{x}}_{k}}$, $\mathbf{G}_{k+1}\doteq\nabla_{\mathbf{x}}g(\mathbf{x})_{\mathbf{x}=\doublehat{\mathbf{x}}_{k+1|k}}$, $\widetilde{\mathbf{F}}^{v}_{k}\doteq\nabla_{\mathbf{v}}\widetilde{f}_{k}(\doublehat{\mathbf{x}}_{k},\mathbf{x}_{k+1},\mathbf{v})|_{\mathbf{v}=\mathbf{0}_{p\times 1}}$, and $\overline{\mathbf{Q}}_{k}=\widetilde{\mathbf{F}}^{v}_{k}\mathbf{R}(\widetilde{\mathbf{F}}^{v}_{k})^{T}$. 


\section{Inverses for Highly non-linear systems}
\label{sec:extensions}
Among various highly nonlinear cases discussed earlier, we focus on the inverses of two-step SOEKF, GS-EKF, and DEKF. In the process, we also derive the inverse of one-step prediction formulation of SOEKF. The one-step prediction formulation is about predicting the next state and analytically more useful while deriving the stability conditions of the filter. Although the two-step prediction-update recursion differs from the one-step prediction formulation in performance and transient behaviour, both have similar convergence properties \cite[Sec.~III]{reif1999stochastic}.

\subsection{Inverse SOEKF}
\label{sec:higher order EKF}
While the EKF formulations are limited to only first-order Taylor series expansion, the forward and inverse filters of SOEKF also include second-order terms. 
Here, we define the Jacobians as $\mathbf{F}_{k}\doteq\nabla_{\mathbf{x}}f(\mathbf{x})\vert_{\mathbf{x}=\hat{\mathbf{x}}_{k}}$ and 
$\mathbf{H}_{k+1}\doteq\nabla_{\mathbf{x}}h(\mathbf{x})\vert_{\mathbf{x}=\hat{\mathbf{x}}_{k+1|k}}$.
\subsubsection{Forward filter}
At each time step, the SOEKF first predicts the current state $\mathbf{x}_{k+1}$ as $\hat{\mathbf{x}}_{k+1|k}$ with the associated prediction error covariance matrix $\bm{\Sigma}_{k+1|k}$ using its previous state estimate. In the update step, the predicted observation estimate $\hat{\mathbf{y}}_{k+1|k}$ and associated error covariance matrix $\mathbf{S}_{k+1}$ are computed based on the predicted state estimate. Finally, the current updated state estimate $\hat{\mathbf{x}}_{k+1}$ and its error covariance matrix $\bm{\Sigma}_{k+1}$ are obtained considering the current observation $\mathbf{y}_{k+1}$ and the predicted estimates. Denoting the $i$-th Euclidean basis vectors in $\mathbb{R}^{n\times 1}$ and $\mathbb{R}^{p\times 1}$ by $\widetilde{\mathbf{a}}_{i}$ and $\mathbf{b}_{i}$, respectively, the forward SOEKF recursions are \cite{bar2004estimation}
\par\noindent\small
\begin{align}
&\textit{Prediction:}\;\hat{\mathbf{x}}_{k+1|k}=f(\hat{\mathbf{x}}_{k})+\frac{1}{2}\sum_{i=1}^{n}\widetilde{\mathbf{a}}_{i}\textrm{Tr}\left(\nabla^{2}\left[f(\hat{\mathbf{x}}_{k})\right]_{i}\bm{\Sigma}_{k}\right),\label{eqn: forward SOEKF x prediction}\\
&\bm{\Sigma}_{k+1|k}=\mathbf{F}_{k}\bm{\Sigma}_{k}\mathbf{F}_{k}^{T}+\mathbf{Q}\nonumber\\
&\hspace{1.5cm}+\frac{1}{2}\sum_{i=1}^{n}\sum_{j=1}^{n}\widetilde{\mathbf{a}}_{i}\widetilde{\mathbf{a}}_{j}^{T}\textrm{Tr}\left(\nabla^{2}\left[f(\hat{\mathbf{x}}_{k})\right]_{i}\bm{\Sigma}_{k}\nabla^{2}\left[f(\hat{\mathbf{x}}_{k})\right]_{j}\bm{\Sigma}_{k}\right),\nonumber\\
&\textit{Update:}&\nonumber\\
&\hat{\mathbf{y}}_{k+1|k}=h(\hat{\mathbf{x}}_{k+1|k})+\frac{1}{2}\sum_{i=1}^{p}\mathbf{b}_{i}\textrm{Tr}\left(\nabla^{2}\left[h(\hat{\mathbf{x}}_{k+1|k})\right]_{i}\bm{\Sigma}_{k+1|k}\right),\label{eqn: forward SOEKF y prediction}\\
&\mathbf{S}_{k+1}=\mathbf{H}_{k+1}\bm{\Sigma}_{k+1|k}\mathbf{H}_{k+1}^{T}+\mathbf{R}\nonumber\\
&\hspace{-0.3cm}+\frac{1}{2}\sum_{i=1}^{p}\sum_{j=1}^{p}\mathbf{b}_{i}\mathbf{b}_{j}^{T}\textrm{Tr}\left(\nabla^{2}\left[h(\hat{\mathbf{x}}_{k+1|k})\right]_{i}\bm{\Sigma}_{k+1|k}\nabla^{2}\left[h(\hat{\mathbf{x}}_{k+1|k})\right]_{j}\bm{\Sigma}_{k+1|k}\right),\nonumber\\
&\mathbf{K}_{k+1}=\bm{\Sigma}_{k+1|k}\mathbf{H}_{k+1}^{T}\mathbf{S}_{k+1}^{-1},\nonumber\\
&\hat{\mathbf{x}}_{k+1}=\hat{\mathbf{x}}_{k+1|k}+\mathbf{K}_{k+1}(\mathbf{y}_{k+1}-\hat{\mathbf{y}}_{k+1|k}),\label{eqn: forward SOEKF x update}\\
&\bm{\Sigma}_{k+1}=\bm{\Sigma}_{k+1|k}-\bm{\Sigma}_{k+1|k}\mathbf{H}_{k+1}^{T}\mathbf{S}_{k+1}^{-1}\mathbf{H}_{k+1}\bm{\Sigma}_{k+1|k},\nonumber
\end{align}
\normalsize
where $\mathbf{K}_{k+1}$ is SOEKF's gain matrix. Unlike EKF, the SOEKF includes Hessian terms in the prediction equations.

\subsubsection{Inverse filter}
Define $\widetilde{f}(\hat{\mathbf{x}},\bm{\Sigma})\doteq f(\hat{\mathbf{x}})+\frac{1}{2}\sum_{i=1}^{n}\widetilde{\mathbf{a}}_{i}\textrm{Tr}\left(\nabla^{2}\left[f(\hat{\mathbf{x}})\right]_{i}\bm{\Sigma}\right)$ and $\widetilde{h}(\hat{\mathbf{x}},\bm{\Sigma})\doteq h(\hat{\mathbf{x}})+\frac{1}{2}\sum_{i=1}^{p}\mathbf{b}_{i}\textrm{Tr}\left(\nabla^{2}\left[h(\hat{\mathbf{x}})\right]_{i}\bm{\Sigma}\right)$. Using \eqref{eqn: observation y} and \eqref{eqn: forward SOEKF x prediction}-
\eqref{eqn: forward SOEKF x update}, the state transition equation for the inverse SOEKF (I-SOEKF) is
\par\noindent\small
\begin{align*}
\hat{\mathbf{x}}_{k+1}&=\overline{f}_{k}(\hat{\mathbf{x}}_{k},\mathbf{x}_{k+1},\mathbf{v}_{k+1})\\
&=\widetilde{f}(\hat{\mathbf{x}}_{k},\bm{\Sigma}_{k})-\mathbf{K}_{k+1}\widetilde{h}(\widetilde{f}(\hat{\mathbf{x}}_{k},\bm{\Sigma_{k}}),\bm{\Sigma}_{k+1|k})\\
&\;\;\;+\mathbf{K}_{k+1}h(\mathbf{x}_{k+1})+\mathbf{K}_{k+1}\mathbf{v}_{k+1}.
\end{align*}
\normalsize

The I-SOEKF recursions now follow directly from forward SOEKF's recursions treating $\mathbf{a}_{k}$ from \eqref{eqn: observation a} as the observation. Denote the Jacobians with respect to the state estimate $\hat{\mathbf{x}}_{k}$ as $\overline{\mathbf{F}}_{k}\doteq\nabla_{\mathbf{x}}\overline{f}_{k}(\mathbf{x},\mathbf{x}_{k+1},\mathbf{0})|_{\mathbf{x}=\doublehat{\mathbf{x}}_{k}}$ and $\mathbf{G}_{k+1}\doteq\nabla_{\mathbf{x}}g(\mathbf{x})|_{\mathbf{x}=\doublehat{\mathbf{x}}_{k+1|k}}$, while the Hessians are denoted as $\nabla^{2}\left[\overline{f}_{k}(\doublehat{\mathbf{x}}_{k})\right]_{i}\doteq\nabla_{x}^{2}\left[\overline{f}_{k}(\mathbf{x},\mathbf{x}_{k+1},\mathbf{0})\right]_{i}|_{\mathbf{x}=\doublehat{\mathbf{x}}_{k}}$ and $\nabla^{2}[g(\doublehat{x}_{k+1|k})]_{i}=\nabla^{2}_{\mathbf{x}}[g(\mathbf{x})]_{i}|_{\mathbf{x}=\doublehat{\mathbf{x}}_{k+1|k}}$. The process noise covariance matrix is $\overline{\mathbf{Q}}_{k}=\overline{\mathbf{V}}_{k}\mathbf{R}\overline{\mathbf{V}}_{k}^{T}$ with $\overline{\mathbf{V}}_{k}\doteq\nabla_{v}\overline{f}_{k}(\doublehat{\mathbf{x}}_{k},\mathbf{x}_{k+1},\mathbf{v})|_{\mathbf{v}=\mathbf{0}}$. The I-SOEKF's state and observation prediction is
\par\noindent\small
\begin{align*}
&\doublehat{\mathbf{x}}_{k+1|k}=\overline{f}_{k}(\doublehat{\mathbf{x}}_{k},\mathbf{x}_{k+1},\mathbf{0})+\frac{1}{2}\sum_{i=1}^{n}\widetilde{\mathbf{a}}_{i}\textrm{Tr}\left(\nabla^{2}\left[\overline{f}_{k}(\doublehat{\mathbf{x}}_{k})\right]_{i}\overline{\bm{\Sigma}}_{k}\right),\\
&\overline{\bm{\Sigma}}_{k+1|k}=\overline{\mathbf{F}}_{k}\overline{\bm{\Sigma}}_{k}\overline{\mathbf{F}}_{k}^{T}+\overline{\mathbf{Q}}_{k}\\
&\hspace{1.0cm}+\frac{1}{2}\sum_{i=1}^{n}\sum_{j=1}^{n}\widetilde{\mathbf{a}}_{i}\widetilde{\mathbf{a}}_{j}^{T}\textrm{Tr}\left(\nabla^{2}\left[\overline{f}_{k}(\doublehat{\mathbf{x}}_{k})\right]_{i}\overline{\bm{\Sigma}}_{k}\nabla^{2}\left[\overline{f}_{k}(\doublehat{\mathbf{x}}_{k})\right]_{j}\overline{\bm{\Sigma}}_{k}\right),\\
&\hat{\mathbf{a}}_{k+1|k}=g(\doublehat{\mathbf{x}}_{k+1|k})+\frac{1}{2}\sum_{i=1}^{n_{a}}\mathbf{d}_{i}\textrm{Tr}\left(\nabla^{2}\left[g(\doublehat{\mathbf{x}}_{k+1|k})\right]_{i}\overline{\bm{\Sigma}}_{k+1|k}\right),\\
&\overline{\mathbf{S}}_{k+1}=\mathbf{G}_{k+1}\overline{\bm{\Sigma}}_{k+1|k}\mathbf{G}_{k+1}^{T}+\bm{\Sigma}_{\epsilon}\\
&\hspace{-0.3cm}+\frac{1}{2}\sum_{i=1}^{n_{a}}\sum_{j=1}^{n_{a}}\mathbf{d}_{i}\mathbf{d}_{j}^{T}\textrm{Tr}\left(\nabla^{2}\left[g(\doublehat{\mathbf{x}}_{k+1|k})\right]_{i}\overline{\bm{\Sigma}}_{k+1|k}\nabla^{2}[g(\doublehat{\mathbf{x}}_{k+1|k})]_{j}\overline{\bm{\Sigma}}_{k+1|k}\right),
\end{align*}
\normalsize
where $\mathbf{d}_{i}$ denotes the $i$-th Euclidean basis vector in $\mathbb{R}^{n_{a}\times 1}$. After these prediction steps, the update steps to compute estimate $\doublehat{\mathbf{x}}_{k+1}$ and $\overline{\bm{\Sigma}}_{k+1}$ are same as \eqref{eqn:IEKF state update} and \eqref{eqn:IEKF covariance update}, respectively.

\begin{remark}
Here, besides the gain matrix $\mathbf{K}_{k+1}$, the Hessian terms are also treated as time-varying parameters of the state transition and approximated by evaluating them using the previous I-SOEKF's estimate $\doublehat{\mathbf{x}}_{k}$.
\end{remark}
\begin{remark}
When the second-order terms are neglected, the forward SOEKF and I-SOEKF reduce to, respectively, forward EKF and I-EKF. 
\end{remark}
\begin{remark}\label{remark:SOEKF complexity}
Similar to I-EKF\cite{singh2022inverse_part1}, I-SOEKF also follows from general SOEKF recursions and hence, has similar computational complexity, i.e., $\mathcal{O}(n^{5})$ where $n$ is the state dimension. Note that a derivative-free efficient implementation \cite{roth2011efficientSOEKF}  of SOEKF yields a $\mathcal{O}(n^{4})$ complexity.
\end{remark}

\subsubsection{One-step prediction formulation}\label{subsubsec: one step SOEKF}
The one-step prediction formulation of SOEKF for the considered system model can be derived following the derivation of two-step recursion formulation of SOEKF as outlined in \cite{bar2004estimation}. Denoting the Jacobians as $\mathbf{F}_{k}\doteq\nabla_{\mathbf{x}}f(\mathbf{x})\vert_{\mathbf{x}=\hat{\mathbf{x}}_{k}}$ and $\mathbf{H}_{k}\doteq\nabla_{\mathbf{x}}h(\mathbf{x})\vert_{\mathbf{x}=\hat{\mathbf{x}}_{k}}$, the one-step prediction SOEKF recursions are
\par\noindent\small
\begin{align}
&\widetilde{\mathbf{x}}_{k}=f(\hat{\mathbf{x}}_{k})+\frac{1}{2}\sum_{i=1}^{n}\widetilde{\mathbf{a}}_{i}\textrm{Tr}\left(\nabla^{2}\left[f(\hat{\mathbf{x}}_{k})\right]_{i}\bm{\Sigma}_{k}\right),\label{eqn: one step SOEKF first update}\\
&\widetilde{\bm{\Sigma}}_{k}=\mathbf{F}_{k}\bm{\Sigma}_{k}\mathbf{F}_{k}^{T}+\mathbf{Q}\nonumber\\
&\hspace{1.0cm}+\frac{1}{2}\sum_{i=1}^{n}\sum_{j=1}^{n}\widetilde{\mathbf{a}}_{i}\widetilde{\mathbf{a}}_{j}^{T}\textrm{Tr}\left(\nabla^{2}\left[f(\hat{\mathbf{x}}_{k})\right]_{i}\bm{\Sigma}_{k}\nabla^{2}\left[f(\hat{\mathbf{x}}_{k})\right]_{j}\bm{\Sigma}_{k}\right),\label{eqn: one step SOEKF sig tilde}\\
&\mathbf{S}_{k}=\mathbf{H}_{k}\bm{\Sigma}_{k}\mathbf{H}_{k}^{T}+\mathbf{R}\nonumber\\
&\hspace{1.0cm}+\frac{1}{2}\sum_{i=1}^{p}\sum_{j=1}^{p}\mathbf{b}_{i}\mathbf{b}_{j}^{T}\textrm{Tr}\left(\nabla^{2}\left[h(\hat{\mathbf{x}}_{k})\right]_{i}\bm{\Sigma}_{k}\nabla^{2}\left[h(\hat{\mathbf{x}}_{k})\right]_{j}\bm{\Sigma}_{k}\right),\\
&\mathbf{M}_{k}=\frac{1}{2}\sum_{i=1}^{n}\sum_{j=1}^{p}\widetilde{\mathbf{a}}_{i}\mathbf{b}_{j}^{T}\textrm{Tr}\left(\nabla^{2}\left[f(\hat{\mathbf{x}}_{k})\right]_{i}\bm{\Sigma}_{k}\nabla^{2}\left[h(\hat{\mathbf{x}}_{k})\right]_{j}\bm{\Sigma}_{k}\right),\\
&\mathbf{K}_{k}=(\mathbf{F}_{k}\bm{\Sigma}_{k}\mathbf{H}_{k}^{T}+\mathbf{M}_{k})\mathbf{S}_{k}^{-1},\label{eqn: one step SOEKF Kk}\\
&\hat{\mathbf{x}}_{k+1}=\widetilde{\mathbf{x}}_{k}+\mathbf{K}_{k}\left(\mathbf{y}_{k}-h(\hat{\mathbf{x}}_{k})-\frac{1}{2}\sum_{i=1}^{p}\mathbf{b}_{i}\textrm{Tr}\left(\nabla^{2}\left[h(\hat{\mathbf{x}}_{k})\right]_{i}\bm{\Sigma}_{k}\right)\right),\label{eqn: one step SOEKF final update}\\
&\bm{\Sigma}_{k+1}=\widetilde{\bm{\Sigma}}_{k}-\mathbf{K}_{k}\mathbf{S}_{k}\mathbf{K}_{k}^{T}.\label{eqn: one step SOEKF sig update}
\end{align}
\normalsize

From \eqref{eqn: observation y}, \eqref{eqn: one step SOEKF first update} and \eqref{eqn: one step SOEKF final update}, the state transition equation for one-step formulation of I-SOEKF is $\hat{\mathbf{x}}_{k+1}\doteq\overline{f}_{k}(\hat{\mathbf{x}}_{k},\mathbf{x}_{k},\mathbf{v}_{k})$ where
\par\noindent\small
\begin{align}
    &\overline{f}_{k}(\hat{\mathbf{x}}_{k},\mathbf{x}_{k},\mathbf{v}_{k})=f(\hat{\mathbf{x}}_{k})-\mathbf{K}_{k}h(\hat{\mathbf{x}}_{k})+\frac{1}{2}\sum_{i=1}^{n}\widetilde{\mathbf{a}}_{i}\textrm{Tr}\left(\nabla^{2}\left[f(\hat{\mathbf{x}}_{k})\right]_{i}\bm{\Sigma}_{k}\right)\nonumber\\
    &\hspace{0.5cm}-\frac{1}{2}\mathbf{K}_{k}\sum_{i=1}^{p}\mathbf{b}_{i}\textrm{Tr}\left(\nabla^{2}\left[h(\hat{\mathbf{x}}_{k})\right]_{i}\bm{\Sigma}_{k}\right)+\mathbf{K}_{k}h(\mathbf{x}_{k})+\mathbf{K}_{k}\mathbf{v}_{k}.\label{eqn: one step I-SOEKF state transition}
\end{align}
\normalsize
Considering the aforementioned state transition and observation \eqref{eqn: observation a}, the inverse filter follows, \textit{mutatis mutandis}, from SOEKF's one-step prediction formulation. The Jacobians with respect to the state estimate are $\overline{\mathbf{F}}_{k}\doteq\nabla_{\mathbf{x}}\overline{f}_{k}(\mathbf{x},\mathbf{x}_{k},\mathbf{0})\vert_{\mathbf{x}=\doublehat{\mathbf{x}}_{k}}=\mathbf{F}_{k}-\mathbf{K}_{k}\mathbf{H}_{k}$ and $\mathbf{G}_{k}\doteq\nabla_{\mathbf{x}}g(\mathbf{x})\vert_{\mathbf{x}=\doublehat{\mathbf{x}}_{k}}$, with the process noise covariance matrix $\overline{\mathbf{Q}}_{k}=\mathbf{K}_{k}\mathbf{R}_{k}\mathbf{K}_{k}^{T}$.

\subsection{Inverse GS-EKF}
\label{sec:gaussian sum ekf}
The GS-EKF \cite{anderson2012optimal} assumes the posterior distribution of the state estimate to be a weighted sum of Gaussian densities, which are updated recursively based on the current observation. This multiple-model filtering technique is also widely used to handle non-Gaussian process and measurement noises in the system model\cite{bilik2010mmse}. Hence, the I-GS-EKF developed in the following is also applicable to non-Gaussian inverse filtering problems.
\subsubsection{Forward filter}
In the forward GS-EKF, we consider $l$ Gaussians denoting the $i$-th Gaussian probability density function, with mean $\overline{\mathbf{x}}_{i,k}$ and covariance $\bm{\Sigma}_{i,k}$ at $k$-th time-step, as $\gamma(\mathbf{x}-\overline{\mathbf{x}}_{i},\bm{\Sigma}_{i})=(2\pi)^{-n/2}|\bm{\Sigma}_{i}|^{-1/2}e^{\left\{-\frac{1}{2}(\mathbf{x}-\overline{\mathbf{x}}_{i})^{T}\bm{\Sigma}_{i}^{-1}(\mathbf{x}-\overline{\mathbf{x}}_{i}) \right\}}$. Given observations $\mathbf{Y}^{k}=\lbrace\mathbf{y}_{j}\rbrace_{1\leq j\leq k}$ up to $k$-th time instant, the posterior distribution of state $\mathbf{x}_{k}$, is approximated as
$p(\mathbf{x}_{k}|\mathbf{Y}^{k})=\sum_{i=1}^{l}c_{i,k}\gamma(\mathbf{x}_{k}-\overline{\mathbf{x}}_{i,k},\bm{\Sigma}_{i,k})$, 
where $c_{i,k}$ is $i$-th Gaussian's weight at $k$-th time step. Linearizing the functions as $\mathbf{F}_{i,k}\doteq\nabla_{\mathbf{x}}f(\mathbf{x})|_{\mathbf{x}=\overline{\mathbf{x}}_{i,k}}$ and $\mathbf{H}_{i,k+1}\doteq\nabla_{\mathbf{x}}h(\mathbf{x})|_{\mathbf{x}=\overline{\mathbf{x}}_{i,k+1|k}}$ for the $i$-th Gaussian, the means $\lbrace\overline{\mathbf{x}}_{i,k}\rbrace_{1\leq i\leq l}$ and covariance matrices $\lbrace\bm{\Sigma}_{i,k}\rbrace_{1\leq i\leq l}$ are updated based on the current observation $\mathbf{y}_{k+1}$ using independent EKF recursions for each Gaussian. Finally, the weights $\lbrace c_{i,k}\rbrace_{1\leq i\leq l}$ are updated as\cite{anderson2012optimal}:
\par\noindent\small
\begin{align*}
c_{i,k+1}=\frac{c_{i,k}\gamma(\mathbf{y}_{k+1}-h(\overline{\mathbf{x}}_{i,k+1|k}),\mathbf{S}_{i,k+1})}{\sum_{i'=1}^{l}c_{i',k}\gamma(\mathbf{y}_{k+1}-h(\overline{\mathbf{x}}_{i',k+1|k}),\mathbf{S}_{i',k+1})},  
\end{align*}
\normalsize
where $\overline{\mathbf{x}}_{i,k+1|k}$ is the $i$-th Gaussian's predicted mean and $\mathbf{S}_{i,k+1}=\mathbf{H}_{i,k+1}\bm{\Sigma}_{i,k+1|k}\mathbf{H}_{i,k+1}^{T}+\mathbf{R}$ with $\bm{\Sigma}_{i,k+1|k}$ as the prediction covariance matrix of $\overline{\mathbf{x}}_{i,k+1|k}$. With these updated Gaussians, the point-estimate $\hat{\mathbf{x}}_{k}$ and the associated covariance matrix $\bm{\Sigma}_{k}$ are 
$\hat{\mathbf{x}}_{k+1}=\sum_{i=1}^{l}c_{i,k+1}\overline{\mathbf{x}}_{i,k+1}$ and 
$\bm{\Sigma}_{k+1}=\sum_{i=1}^{l}c_{i,k+1}\left(\bm{\Sigma}_{i,k+1}+(\hat{\mathbf{x}}_{k+1}-\overline{\mathbf{x}}_{i,k+1})(\hat{\mathbf{x}}_{k+1}-\overline{\mathbf{x}}_{i,k+1})^{T}\right)$. 

\subsubsection{Inverse filter}
Consider an augmented state vector $\mathbf{z}_{k}=\lbrace\lbrace\overline{\mathbf{x}}_{i,k}\rbrace_{1\leq i\leq l},\lbrace c_{i,k}\rbrace_{1\leq i\leq l}\rbrace$ (means and weights of forward GS-EKF). Then, substituting for observation $\mathbf{y}_{k+1}$ from \eqref{eqn: observation y} in the forward filter's updates similar to I-SOEKF and denoting the $i$-th EKF's gain matrix as $\mathbf{K}_{i,k+1}=\bm{\Sigma}_{i,k+1|k}\mathbf{H}_{i,k+1}^{T}\mathbf{S}_{i,k+1}^{-1}$ yields the state transition equations for inverse GS-EKF (I-GS-EKF) as
\par\noindent\small
\begin{align}
&\overline{\mathbf{x}}_{i,k+1}=f(\overline{\mathbf{x}}_{i,k})+\mathbf{K}_{i,k+1}\left(h(\mathbf{x}_{k+1})+\mathbf{v}_{k+1}-h(f(\overline{\mathbf{x}}_{i,k}))\right),\nonumber\\
&c_{i,k+1}=\frac{c_{i,k}\gamma(h(\mathbf{x}_{k+1})+\mathbf{v}_{k+1}-h(f(\overline{\mathbf{x}}_{i,k})),\mathbf{S}_{i,k+1})}{\sum_{i'=1}^{l}c_{i',k}\gamma(h(\mathbf{x}_{k+1})+\mathbf{v}_{k+1}-h(f(\overline{\mathbf{x}}_{i',k})),\mathbf{S}_{i',k+1})}.\label{eqn: inverse GS-EKF weight transition}
\end{align}
\normalsize
Treating $\lbrace\mathbf{K}_{i,k+1},\mathbf{S}_{i,k+1}\rbrace_{1\leq i\leq l}$ as time-varying parameters of the state transition equations, which are approximated in a similar way as we approximated the gain matrices for I-EKF and I-SOEKF, the overall state transition in terms of the augmented state is 
$\mathbf{z}_{k+1}=\overline{f}_{k}(\mathbf{z}_{k},\mathbf{x}_{k+1},\mathbf{v}_{k+1})$.
Similarly, the observation $\mathbf{a}_{k}$ as a function of augmented state $\mathbf{z}_{k}$ is 
$\mathbf{a}_{k}=\overline{g}(\mathbf{z}_{k})+\bm{\epsilon}_{k}=g\left(\sum_{i=1}^{l}c_{i,k}\overline{\mathbf{x}}_{i,k}\right)+\bm{\epsilon}_{k}$. 

The I-GS-EKF approximates the posterior distribution of the augmented state as a sum of $\overline{l}$ Gaussians with its recursions again following directly from forward GS-EKF's recursions treating $\mathbf{a}_{k}$ as the observation. However, the inverse filter estimates an $l(n+1)$-dimensional augmented state $\mathbf{z}_{k}$ with the Jacobians with respect to the state denoted as $\overline{\mathbf{F}}_{j,k}\doteq\nabla_{\mathbf{z}}\overline{f}_{k}(\mathbf{z},\mathbf{x}_{k+1},\mathbf{0})|_{\mathbf{z}=\overline{\mathbf{z}}_{j,k}}$ and $\mathbf{G}_{j,k+1}\doteq\nabla_{\mathbf{z}}\overline{g}(\mathbf{z})|_{\mathbf{z}=\overline{\mathbf{z}}_{j,k+1|k}}$, and the process noise covariance matrix as $\overline{\mathbf{Q}}_{k}=\overline{\mathbf{V}}_{j,k}\mathbf{R}\overline{\mathbf{V}}_{j,k}^{T}$ with $\overline{\mathbf{V}}_{j,k}\doteq\nabla_{\mathbf{v}}\overline{f}_{k}(\overline{\mathbf{z}}_{j,k},\mathbf{x}_{k+1},\mathbf{v})|_{\mathbf{v}=\mathbf{0}}$ for the $j$-th inverse filter's Gaussian updates. The point estimate $\hat{\mathbf{z}}_{k}$ consists of the estimates $\lbrace\hat{\overline{\mathbf{x}}}_{i,k},\hat{c}_{i,k}\rbrace_{1\leq i\leq l}$ of the forward GS-EKF's means and weights such that the point estimate $\doublehat{\mathbf{x}}_{k}$ of the forward filter's estimate $\hat{\mathbf{x}}_{k}$ is $\doublehat{\mathbf{x}}_{k}=\sum_{i=1}^{l}\hat{c}_{i,k}\hat{\overline{\mathbf{x}}}_{i,k}$.

\begin{remark}
When the forward filter considers only one Gaussian ($l=1$), the forward GS-EKF reduces to forward EKF with the only weight $c_{1,k}=1$ for all $k$. Hence, this weight need not be considered in the augmented state and $\mathbf{z}_{k}$ reduces to $\overline{\mathbf{x}}_{1,k}$ which is the estimate $\hat{\mathbf{x}}_{k}$ itself. Similarly, I-GS-EKF also reduces to I-EKF if only one Gaussian is considered ($\overline{l}=1$).
\end{remark}
\begin{remark}\label{remark: GSEKF complexity}
Since GS-EKF consists of several independent EKF, its computational load is larger than an EKF considering only one Gaussian. Furthermore, I-GS-EKF considers an augmented state of dimension $l(n+1)$, which is greater than that of the forward GS-EKF. Hence, in general, I-GS-EKF is computationally more complex than the forward GS-EKF and the I-EKF. However, our numerical experiments in Section~\ref{subsec:sim EKF SOEKF GS-EKF} suggest that I-GS-EKF can provide reasonable accuracy even when considering a smaller number of Gaussians than that in forward GS-EKF, i.e., when $\overline{l}<l$. A survey of several methods to reduce the computational complexity of these Kalman filter extensions is available in \cite{raitoharju2019computational}.
\end{remark}

\subsection{Inverse DEKF}
\label{sec:dithered EKF}
Consider the adversary employing DEKF \cite{weiss1980improved} as its forward filter. In DEKF, the output non-linearities are modified using dither signals so as to tighten the cone-bounds. 
Dithering tightens this cone such that the non-linearities are smoothened but it may also degrade the near-optimal performance of the EKF after the initial transient phase of estimation. 
Therefore, 
dithering is introduced only during the initial transient phase with the aim to improve the filter's transient performance and avoid divergence. Denote the dither amplitude which controls the tightness of the cone-bounds by $d$ and its amplitude probability density function by $p(a)$. The observation function $h(x)$ is dithered as 
$h^{*}(x)=\int_{-d}^{d}h(a+x)p(a)da$. 
If $d=d_{0}e^{-k/\tau}$, where $d_{0}$ and $\tau$ are constants and `$k$' denotes the time index, then $h^{*}(x)\to h(x)$ exponentially as $k\to\infty$ during the transient phase.

The forward DEKF follows from conventional forward EKF described in Section \ref{sec:background} (detailed in the companion paper (Part I)\cite[Section~III-C]{singh2022inverse_part1}) by replacing $h(\cdot)$ with $h^{*}(\cdot)$ as the observation function of $\mathbf{y}_{k}$ during the initial transient phase and hence, the inverse DEKF (I-DEKF) also follows from I-EKF restated in Section \ref{sec:background}. The dither of the adversary's filter is assumed to be known to us. 
Otherwise, 
the I-DEKF may also proceed with the unmodified observation function. We show in Section~\ref{subsec:sim dithered EKF} that these two formulations, labeled I-DEKF-1 and I-DEKF-2, respectively, generally vary in their estimation performances especially during the transient phase where the modified observation function is considered.

\begin{remark}\label{remark: DEKF complexity}
A DEKF differs from a standard EKF only because of the modified observation function and hence, has the same $\mathcal{O}(n^{3})$ computational complexity as EKF. The same argument holds for I-DEKF, which also follows from standard I-EKF recursions.
\end{remark}

\section{Stability Analyses}
\label{sec:stability}
Similar to I-EKF stability analysis in the companion paper (Part I) \cite{singh2022inverse_part1}, we first obtain stability conditions of forward SOEKF in the exponential-boundedness-mean-squared sense and then, extend those results to I-SOEKF's stability. In particular, we adopt the bounded non-linearity approach \cite{reif1999stochastic} and consider the one-step prediction formulation of forward and inverse SOEKF derived in Section~\ref{subsubsec: one step SOEKF}. As shown in \cite[Sec. III]{reif1999stochastic}, the convergence behaviors of both one- and two-step formulations are similar. In the following, we rely on the definitions of the exponential-boundedness and boundedness-with-probability-one of a stochastic process as well as a useful lemma from \cite{reif1999stochastic} that were restated in the companion paper (Part I) \cite[Definition 1 and 2, and Lemma 1]{singh2022inverse_part1}.

\subsection{Forward SOEKF stability}
\label{subsec: Stability of SOEKF}
Consider the one-step SOEKF's formulation \eqref{eqn: one step SOEKF first update}-\eqref{eqn: one step SOEKF sig update}. Considering second-order terms as well, the Taylor series expansion of functions $f(\cdot)$ and $h(\cdot)$ at the estimate $\hat{\mathbf{x}}_{k}$ are
 \par\noindent\small
\begin{align*}
&f(\mathbf{x}_{k})-f(\hat{\mathbf{x}}_{k})=\mathbf{F}_{k}(\mathbf{x}_{k}-\hat{\mathbf{x}}_{k})\\
&\;\;+\frac{1}{2}\sum_{i=1}^{n}\widetilde{\mathbf{a}}_{i}(\mathbf{x}_{k}-\hat{\mathbf{x}}_{k})^{T}\nabla^{2}\left[f(\hat{\mathbf{x}}_{k})\right]_{i}(\mathbf{x}_{k}-\hat{\mathbf{x}}_{k})+\phi(\mathbf{x}_{k},\hat{\mathbf{x}}_{k}),\\
&h(\mathbf{x}_{k})-h(\hat{\mathbf{x}}_{k})=\mathbf{H}_{k}(\mathbf{x}_{k}-\hat{\mathbf{x}}_{k})\\
&\;\;+\frac{1}{2}\sum_{i=1}^{p}\mathbf{b}_{i}(\mathbf{x}_{k}-\hat{\mathbf{x}}_{k})^{T}\nabla^{2}\left[h(\hat{\mathbf{x}}_{k})\right]_{i}(\mathbf{x}_{k}-\hat{\mathbf{x}}_{k})+\chi(\mathbf{x}_{k},\hat{\mathbf{x}}_{k}),
\end{align*}
\normalsize
where $\phi(\cdot)$ and $\chi(\cdot)$ are suitable non-linear functions to account for third and higher-order terms in the expansions. Using these expansions, the error dynamics of the forward filter with $\mathbf{e}_{k}\doteq\mathbf{x}_{k}-\hat{\mathbf{x}}_{k}$ is
\par\noindent\small
\begin{align}
\mathbf{e}_{k+1}=(\mathbf{F}_{k}-\mathbf{K}_{k}\mathbf{H}_{k})\mathbf{e}_{k}+\mathbf{r}_{k}+\mathbf{q}_{k}+\mathbf{s}_{k},\label{eqn: forward SOEKF error}
\end{align}
\normalsize
where
\par\noindent\small
\begin{align*}
\mathbf{r}_{k}&=\phi(\mathbf{x}_{k},\hat{\mathbf{x}}_{k})-\mathbf{K}_{k}\chi(\mathbf{x}_{k},\hat{\mathbf{x}}_{k}),\\
\mathbf{q}_{k}&=\frac{1}{2}\sum_{i=1}^{n}\widetilde{\mathbf{a}}_{i}\mathbf{e}_{k}^{T}\nabla^{2}\left[f(\hat{\mathbf{x}}_{k})\right]_{i}\mathbf{e}_{k}-\frac{1}{2}\sum_{i=1}^{n}\widetilde{\mathbf{a}}_{i}\textrm{Tr}\left(\nabla^{2}\left[f(\hat{\mathbf{x}}_{k})\right]_{i}\bm{\Sigma}_{k}\right)\\
&-\frac{1}{2}\mathbf{K}_{k}\sum_{i=1}^{p}\mathbf{b}_{i}\mathbf{e}_{k}^{T}\nabla^{2}\left[h(\hat{\mathbf{x}}_{k})\right]_{i}\mathbf{e}_{k}+\frac{1}{2}\mathbf{K}_{k}\sum_{i=1}^{p}\mathbf{b}_{i}\textrm{Tr}\left(\nabla^{2}\left[h(\hat{\mathbf{x}}_{k})\right]_{i}\bm{\Sigma}_{k}\right),\\
\mathbf{s}_{k}&=\mathbf{w}_{k}-\mathbf{K}_{k}\mathbf{v}_{k}.
\end{align*}
\normalsize

The following Theorem~\ref{theorem: forward SOEKF stability} provides sufficient conditions for the stochastic stability of forward SOEKF.
\begin{theorem}[Exponential boundedness of forward SOEKF's error]
\label{theorem: forward SOEKF stability}
Consider the non-linear stochastic system defined by \eqref{eqn: state transition x} and \eqref{eqn: observation y}, and SOEKF's one-step prediction formulation \eqref{eqn: one step SOEKF first update}-\eqref{eqn: one step SOEKF sig update}. Let the following assumptions hold true.\\
\textit{1)} There exist positive real numbers $\overline{f}$, $\overline{h}$, $\underline{\sigma}$, $\overline{\sigma}$, $\underline{q}$, $\underline{r}$, $\overline{a}$, $\overline{b}$, $\delta$ and real numbers $\underline{a}$, $\underline{b}$ (not necessarily positive) such that the following bounds are satisfied for all $k\geq 0$.
    \par\noindent\small
    \begin{align*}
    \underline{\sigma}\mathbf{I}&\preceq\bm{\Sigma}_{k}\preceq\overline{\sigma}\mathbf{I},\hspace{0.3cm}\|\mathbf{F}_{k}\|\leq\overline{f},\hspace{0.3cm}\|\mathbf{H}_{k}\|\leq\overline{h},\\
    \underline{r}\mathbf{I}&\preceq \mathbf{R}_{k}\preceq\delta\mathbf{I},\hspace{0.3cm}\underline{q}\mathbf{I}\preceq\mathbf{Q}_{k}\preceq\delta\mathbf{I},\\
    \underline{a}\mathbf{I}&\preceq\nabla^{2}\left[f(\hat{\mathbf{x}}_{k})\right]_{i}\preceq\overline{a}\mathbf{I}\hspace{0.2cm}\forall i\in\lbrace1,2,\hdots,n\rbrace,\\
    \underline{b}\mathbf{I}&\preceq\nabla^{2}\left[h(\hat{\mathbf{x}}_{k})\right]_{j}\preceq\overline{b}\mathbf{I}\hspace{0.2cm}\forall j\in\lbrace1,2,\hdots,p\rbrace.
    \end{align*}
    \normalsize
\textit{2)} $\mathbf{F}_{k}$ is non-singular and $\mathbf{F}_{k}^{-1}$ satisfies the following bound for all $k\geq 0$ for some positive real number $\widetilde{f}$,
    \par\noindent\small
    \begin{align*}
        \|\mathbf{F}_{k}^{-1}\|\leq\widetilde{f}.
    \end{align*}
    \normalsize
\textit{3)} There exist positive real numbers $\kappa_{\phi}$, $\epsilon_{\phi}$, $\kappa_{\chi}$, $\epsilon_{\chi}$ such that the non-linear functions $\phi(\cdot)$ and $\chi(\cdot)$ satisfy
    \par\noindent\small
    \begin{align*}
        &\|\phi(\mathbf{x},\hat{\mathbf{x}})\|_{2}\leq\kappa_{\phi}\|\mathbf{x}-\hat{\mathbf{x}}\|_{2}^{3}\hspace{0.5cm}for\hspace{0.5cm}\|\mathbf{x}-\hat{\mathbf{x}}\|_{2}\leq\epsilon_{\phi},\\
        &\|\chi(\mathbf{x},\hat{\mathbf{x}})\|_{2}\leq\kappa_{\chi}\|\mathbf{x}-\hat{\mathbf{x}}\|_{2}^{3}\hspace{0.5cm}for\hspace{0.5cm}\|\mathbf{x}-\hat{\mathbf{x}}\|_{2}\leq\epsilon_{\chi}.
    \end{align*}
    \normalsize
Then the estimation error given by \eqref{eqn: forward SOEKF error} is exponentially bounded in mean-squared sense if the estimation error is within $\epsilon$ bound for a suitable constant $\epsilon>0$, 
\par\noindent\small
\begin{align}
&\widetilde{f}<\frac{2\underline{r}}{\overline{h}\overline{a}\overline{b}\overline{\sigma}^{2}n\sqrt{np}},\label{eqn: SOEKF stable constraint on inverse norm}\\
&\underline{q}> c,\label{eqn: SOEKF stable constraint on q}
\end{align}
and
\begin{align}
&\delta=\frac{1}{\kappa_{\textrm{noise}}}\left(\frac{\alpha\widetilde{\epsilon}^{2}}{2\overline{\sigma}}-c_{\textrm{sec}}\right),\label{eqn: SOEKF stable constraint on delta}
\end{align}
\normalsize
for some $\widetilde{\epsilon}<\epsilon$, where $c$, $\alpha$, $\kappa_{\textrm{noise}}$ and $c_{\textrm{sec}}$ are constants that depend on the bounds assumed on the system.
\end{theorem}
\begin{proof}
See Appendix~\ref{App-thm-forward SOEKF stability}.
\end{proof}

\begin{remark}
It follows from the proof of Lemma~\ref{lemma: SOEKF stable alpha term} (see Appendix~\ref{App-thm-forward SOEKF stability}) that the constant $c$ depends on $\delta$. The conditions \eqref{eqn: SOEKF stable constraint on q} and \eqref{eqn: SOEKF stable constraint on delta} suggest that $\delta$ should be chosen appropriately so that both the conditions could be satisfied simultaneously. However, the exact bounds on $\delta$ depend on other bounds assumed on the system matrices in Theorem~\ref{theorem: forward SOEKF stability}.
\end{remark}
\begin{remark}
Furthermore, as with the bounded non-linearity approach for EKF, these bounds may be conservative and, consequently, the estimation error may remain bounded outside this range \cite[Sec. V]{reif1999stochastic}.
\end{remark}

\subsection{Inverse SOEKF stability}
\label{subsec: Stability of inverse SOEKF}
Considering a suitable non-linear function $\overline{\chi}(\cdot)$, the Taylor series expansion of $g(\cdot)$ at estimate $\doublehat{\mathbf{x}}_{k}$ of I-SOEKF's one-step prediction formulation is
\par\noindent\small
\begin{align*}
&g(\hat{\mathbf{x}}_{k})-g(\doublehat{\mathbf{x}}_{k})=\mathbf{G}_{k}(\hat{\mathbf{x}}_{k}-\doublehat{\mathbf{x}}_{k})\\
&\;\;\;+\frac{1}{2}\sum_{i=1}^{n_{a}}\mathbf{d}_{i}(\hat{\mathbf{x}}_{k}-\doublehat{\mathbf{x}}_{k})^{T}\nabla^{2}\left[g(\doublehat{\mathbf{x}}_{k})\right]_{i}(\hat{\mathbf{x}}_{k}-\doublehat{\mathbf{x}}_{k})+\overline{\chi}(\hat{\mathbf{x}}_{k},\doublehat{\mathbf{x}}_{k}),
\end{align*}
\normalsize
where $\mathbf{d}_{i}$ is the $i$-th Euclidean basis vector in $\mathbb{R}^{n_{a}\times 1}$. Finally, the error dynamics of the inverse filter with estimation error denoted by $\overline{\mathbf{e}}_{k}\doteq\hat{\mathbf{x}}_{k}-\doublehat{\mathbf{x}}_{k}$ and the inverse filter's Kalman gain and estimation error covariance matrix by $\overline{\mathbf{K}}_{k}$ and $\overline{\bm{\Sigma}}_{k}$, respectively, is
\par\noindent\small
\begin{align}
    \overline{\mathbf{e}}_{k+1}=(\overline{\mathbf{F}}_{k}-\overline{\mathbf{K}}_{k}\mathbf{G}_{k})\overline{\mathbf{e}}_{k}+\overline{\mathbf{r}}_{k}+\overline{\mathbf{q}}_{k}+\overline{\mathbf{s}}_{k},\label{eqn: inverse SOEKF error}
\end{align}
\normalsize
where
\par\noindent\small
\begin{align*}
&\overline{\mathbf{r}}_{k}=\overline{\phi}_{k}(\hat{\mathbf{x}}_{k},\doublehat{\mathbf{x}}_{k})-\overline{\mathbf{K}}_{k}\overline{\chi}(\hat{\mathbf{x}}_{k},\doublehat{\mathbf{x}}_{k}),\\
&\overline{\mathbf{q}}_{k}=\frac{1}{2}\sum_{i=1}^{n}\widetilde{\mathbf{a}}_{i}\overline{\mathbf{e}}_{k}^{T}\nabla^{2}\left[\overline{f}_{k}(\doublehat{\mathbf{x}}_{k})\right]_{i}\overline{\mathbf{e}}_{k}-\frac{1}{2}\sum_{i=1}^{n}\widetilde{\mathbf{a}}_{i}\textrm{Tr}\left(\nabla^{2}\left[\overline{f}_{k}(\doublehat{\mathbf{x}}_{k})\right]_{i}\overline{\bm{\Sigma}}_{k}\right)\\
&\hspace{-0.3cm}-\frac{1}{2}\overline{\mathbf{K}}_{k}\sum_{i=1}^{n_{a}}\mathbf{d}_{i}\overline{\mathbf{e}}_{k}^{T}\nabla^{2}\left[g(\doublehat{\mathbf{x}}_{k})\right]_{i}\overline{\mathbf{e}}_{k}+\frac{1}{2}\overline{\mathbf{K}}_{k}\sum_{i=1}^{n_{a}}\mathbf{d}_{i}\textrm{Tr}\left(\nabla^{2}\left[g(\doublehat{\mathbf{x}}_{k})\right]_{i}\overline{\bm{\Sigma}}_{k}\right),\\
&\overline{\mathbf{s}}_{k}=\mathbf{K}_{k}\mathbf{v}_{k}-\overline{\mathbf{K}}_{k}\bm{\epsilon}_{k},
\end{align*}
\normalsize
with $\overline{\phi}_{k}(\hat{\mathbf{x}}_{k},\doublehat{\mathbf{x}}_{k})=\phi(\hat{\mathbf{x}}_{k},\doublehat{\mathbf{x}}_{k})-\mathbf{K}_{k}\chi(\hat{\mathbf{x}}_{k},\doublehat{\mathbf{x}}_{k})$ and $\frac{1}{2}\sum_{i=1}^{n}\widetilde{\mathbf{a}}_{i}\overline{\mathbf{e}}_{k}^{T}\nabla^{2}\left[\overline{f}_{k}(\doublehat{\mathbf{x}}_{k})\right]_{i}\overline{\mathbf{e}}_{k}=\frac{1}{2}\sum_{i=1}^{n}\widetilde{\mathbf{a}}_{i}\overline{\mathbf{e}}_{k}^{T}\nabla^{2}\left[f(\doublehat{\mathbf{x}}_{k})\right]_{i}\overline{\mathbf{e}}_{k}-\frac{1}{2}\mathbf{K}_{k}\sum_{i=1}^{p}\mathbf{b}_{i}\overline{\mathbf{e}}_{k}^{T}\nabla^{2}\left[h(\doublehat{\mathbf{x}}_{k})\right]_{i}\overline{\mathbf{e}}_{k}$.
Here, the error in approximations of the terms $\frac{1}{2}\sum_{i=1}^{n}\widetilde{\mathbf{a}}_{i}\textrm{Tr}\left(\nabla^{2}\left[f(\hat{\mathbf{x}}_{k})\right]_{i}\bm{\Sigma}_{k}\right)$ and $\frac{1}{2}\mathbf{K}_{k}\sum_{i=1}^{p}\mathbf{b}_{i}\textrm{Tr}\left(\nabla^{2}\left[h(\hat{\mathbf{x}}_{k})\right]_{i}\bm{\Sigma}_{k}\right)$ by I-SOEKF (as mentioned in Section \ref{sec:higher order EKF}) are neglected. Also, using the bounds assumed in Theorem \ref{theorem: forward SOEKF stability}, these approximation errors are bounded by positive constants.

The following Theorem~\ref{theorem: inverse SOEKF stability} states the conditions for stability of I-SOEKF.

\begin{theorem}[Exponential boundedness of I-SOEKF's error]
\label{theorem: inverse SOEKF stability}
Consider the adversary's forward SOEKF's one-step prediction formulation that is stable as per Theorem \ref{theorem: forward SOEKF stability}. Additionally, assume that the following hold true.\\
\textit{1)} There exist positive real numbers $\overline{g}$, $\underline{m}$, $\overline{m}$, $\underline{\epsilon}$, $\overline{c}$, $\overline{\delta}$ and a real number $\underline{c}$ (not necessarily positive) such that the following bounds are fulfilled for all $k\geq 0$.
    \par\noindent\small
    \begin{align*}
        &\|\mathbf{G}_{k}\|\leq\overline{g},\hspace{0.3cm}
        \underline{m}\mathbf{I}\preceq\overline{\bm{\Sigma}}_{k}\preceq\overline{m}\mathbf{I},\hspace{0.3cm}\underline{\epsilon}\mathbf{I}\preceq\overline{\mathbf{R}}_{k}\preceq\overline{\delta}\mathbf{I},\\
        &\underline{c}\mathbf{I}\preceq\nabla^{2}\left[g(\doublehat{\mathbf{x}}_{k})\right]_{i}\preceq\overline{c}\mathbf{I}\hspace{0.5cm}\forall i\in\lbrace 1,2,\hdots,n_{a}\rbrace.
    \end{align*}
    \normalsize
\textit{2)} $\mathbf{H}_{k}\bm{\Sigma}_{k}\mathbf{F}_{k}^{T}+\mathbf{M}_{k}^{T}$ is full column rank matrix for every $k\geq 0$.
\textit{3)} There exist positive real numbers $\kappa_{\bar{\chi}}$, $\epsilon_{\bar{\chi}}$ such that the non-linear function $\overline{\chi}(\cdot)$ satisfies
    \par\noindent\small
    \begin{align*}
        \|\overline{\chi}(\hat{\mathbf{x}},\doublehat{\mathbf{x}})\|_{2}\leq\kappa_{\bar{\chi}}\|\hat{\mathbf{x}}-\doublehat{\mathbf{x}}\|_{2}^{3}\hspace{0.5cm}for\hspace{0.5cm}\|\hat{\mathbf{x}}-\doublehat{\mathbf{x}}\|_{2}\leq\epsilon_{\bar{\chi}}.
    \end{align*}
    \normalsize
Then, the estimation error of I-SOEKF given by \eqref{eqn: inverse SOEKF error} is exponentially bounded in mean-squared sense if the estimation error is bounded by a suitable constant $\overline{\epsilon}>0$ and the bound constants also satisfy the equivalent conditions of \eqref{eqn: SOEKF stable constraint on inverse norm}, \eqref{eqn: SOEKF stable constraint on q}, and \eqref{eqn: SOEKF stable constraint on delta} for the inverse filter dynamics.
\end{theorem}

\begin{proof}
See Appendix~\ref{App-thm-inverse SOEKF stability}.
\end{proof}

\begin{remark}\label{remark:SOEKF observability}
As discussed in the companion paper (Part I)\cite[Sec.~V-C]{singh2022inverse_part1}, the inequality $\underline{m}\mathbf{I}\preceq\overline{\bm{\Sigma}}_{k}\preceq\overline{m}\mathbf{I}$ assumed in Theorem~\ref{theorem: inverse SOEKF stability} is closely related to the theoretical observability of the non-linear inverse filtering model. Further, the inequality $\underline{m}\mathbf{I}\preceq\overline{\bm{\Sigma}}_{k}\preceq\overline{m}\mathbf{I}$ is a weaker assumption than observability.
\end{remark}
\begin{remark}\label{remark:consistency SOEKF GS-EKF}
Furthermore, the I-EKF's consistency results of \cite[Theorem~6]{singh2022inverse_part1} can be trivially extended to I-SOEKF's and I-GS-EKF's state estimates. Note that a GS-EKF's state estimate is a weighted sum of a finite number of independent EKF recursions, which are consistent as shown in \cite[Theorem~6]{singh2022inverse_part1}.
\end{remark}

\section{Unknown Forward Filter and Model}
\label{sec:RKHS-EKF}
In practice, the defender's state transition function $f(\cdot)$ of \eqref{eqn: state transition x} may not be known to the adversary. Similarly, the defender may not know the specific forward filter employed by the adversary and the adversary's function $g(\cdot)$ of \eqref{eqn: observation a} for computing its actions. In this context, we propose a recursive RKHS-EKF to jointly compute the desired state estimates and learn the unknown functions. In particular, we approximate an unknown function $s(\cdot):\mathbb{R}^{n}\to\mathbb{R}$ by a function in RKHS induced by a kernel $\kappa(\cdot,\cdot):\mathbb{R}^{n}\times\mathbb{R}^{n}\to\mathbb{R}$\cite{aronszajn1950theory}. The RKHS-based function approximation has been proposed for non-linear state-space modeling\cite{tobar2015unsupervised,ralaivola2003dynamical} and recursive least-squares algorithms with non-linear unknown functions\cite{engel2004kernel,weifeng2009extended_krls,van2006sliding}. The representer theorem \cite{scholkopf2001generalized} ensures that the optimal approximation in RKHS with respect to an arbitrary loss function takes the form $s(\cdot)\approx\sum_{i=1}^{M}a_{i}\kappa(\widetilde{\mathbf{x}}_{i},\cdot)$, where $\{\widetilde{\mathbf{x}}_{i}\}_{1\leq i\leq M}$ are the $M$ input training samples or dictionary and $\{a_{i}\}_{1\leq i\leq M}$ are the corresponding mixing parameters to be learnt. Often, a Gaussian kernel such as $\kappa(\mathbf{x}_{i},\mathbf{x}_{j})=\exp{\left(-\frac{\|\mathbf{x}_{i}-\mathbf{x}_{j}\|^{2}_{2}}{\sigma^2}\right)}$, with kernel width $\sigma>0$ controlling the smoothness of the approximation, is commonly used. It is a universal kernel \cite{steinwart2001influence}, i.e., its induced RKHS is dense in the space of continuous functions. 

In the following, we consider a general non-linear system model, wherein both state transition and observation functions are unknown to the agent employing the filter, and develop a general RKHS-EKF. The defender can employ the RKHS-EKF as its inverse filter without assuming any prior information about the adversary's forward filter. Our RKHS-EKF may be trivially simplified to yield forward RKHS-EKF for the adversary which knows its observation function. In particular, our RKHS-EKF adopts EKF to obtain the state estimates while the unknown system parameters are learnt using EM algorithm\cite{hajek2015random}. The EM algorithm is widely used to compute maximum likelihood estimates in presence of missing data.
 
\textit{System models for uncertain dynamics:} We examine the non-linear state transition \eqref{eqn: state transition x} and observation \eqref{eqn: observation y} with the functions $f(\cdot)$ and $h(\cdot)$ as well as the noise covariances $\mathbf{Q}$ and $\mathbf{R}$ being unknown. Consider a kernel function $\kappa(\cdot,\cdot)$ and a dictionary $\{\widetilde{\mathbf{x}}_{l}\}_{1\leq l\leq L}$ of size $L$. Denote $\bm{\Phi}(\mathbf{x})=\begin{bmatrix}\kappa(\widetilde{\mathbf{x}}_{1},\mathbf{x})&\hdots&\kappa(\widetilde{\mathbf{x}}_{L},\mathbf{x})\end{bmatrix}^{T}$. Using the kernel function approximation, the filter's state transition and observation are approximated, respectively, as
\par\noindent\small
\begin{align}
    &\mathbf{x}_{k+1}=\mathbf{A}\bm{\Phi}(\mathbf{x}_{k})+\mathbf{w}_{k},\label{eqn:RKHS-EKF state transition approx}\\
    &\mathbf{y}_{k}=\mathbf{B}\bm{\Phi}(\mathbf{x}_{k})+\mathbf{v}_{k},\label{eqn:RKHS-EKF observation approx}
\end{align}
\normalsize
where the coefficient matrices $\mathbf{A}\in\mathbb{R}^{n\times L}$ and $\mathbf{B}\in\mathbb{R}^{p\times L}$ consists of the unknown mixing parameters to be learnt. The dictionary $\{\widetilde{\mathbf{x}}_{l}\}_{1\leq l\leq L}$ can be formed using a sliding window\cite{van2006sliding} or approximate linear dependency (ALD)\cite{engel2004kernel} criterion. Besides the state estimate $\hat{\mathbf{x}}_{k}$, the RKHS-EKF needs to estimate the unknown parameters $\Theta=\{\mathbf{A},\mathbf{B},\mathbf{Q},\mathbf{R}\}$ based on the observations $\{\mathbf{y}_{i}\}_{1\leq i\leq k}$ upto $k$-th time step.

In Section~\ref{subsec: parameter learning}, we derive an approximate online EM to obtain the parameter estimates, which is then coupled with EKF to formulate the RKHS-EKF recursions in Section~\ref{subsec: recursion}. Finally, the RKHS-EKF to jointly estimate the state and system parameters is summarized in Algorithms~\ref{alg:RKHS-EKF initialization} and \ref{alg:RKHS-EKF recursion}.

\subsection{
Parameter learning}\label{subsec: parameter learning}
Given $\Theta$, the state estimate $\hat{\mathbf{x}}_{k}$ can be computed using EKF-based recursions. To estimate the unknown parameters $\Theta$, we consider EM algorithm\cite{hajek2015random}.

\textit{EM for parameter learning:} Consider the states upto time $k$ as $\mathbf{X}^{k}=\{\mathbf{x}_{j}\}_{0\leq j\leq k}$ and the corresponding observations $\mathbf{Y}^{k}=\{\mathbf{y}_{j}\}_{1\leq j\leq k}$. The joint conditional probability density given the parameters $\Theta$ is
\par\noindent\small
\begin{align}
    p(\mathbf{X}^{k},\mathbf{Y}^{k}|\Theta)=p(\mathbf{x}_{0})\prod_{j=1}^{k}p(\mathbf{x}_{j}|\mathbf{x}_{j-1},\Theta)\prod_{j=1}^{k}p(\mathbf{y}_{j}|\mathbf{x}_{j},\Theta).\label{eqn:RKHS-EKF joint density}
\end{align}
\normalsize

Under the Gaussian noise assumption, the conditional probability densities are
\par\noindent\small
\begin{align}
    &p(\mathbf{x}_{j}|\mathbf{x}_{j-1},\Theta)=\gamma(\mathbf{x}_{j}-\mathbf{A}\bm{\Phi}(\mathbf{x}_{j-1}),\mathbf{Q}),\label{eqn:RKHS-EKF conditional density x}\\
    &p(\mathbf{y}_{j}|\mathbf{x}_{j},\Theta)=\gamma(\mathbf{y}_{j}-\mathbf{B}\bm{\Phi}(\mathbf{x}_{j}),\mathbf{R}),\label{eqn:RKHS-EKF conditional density y}
\end{align}
\normalsize
with the Gaussian probability density function $\gamma(\cdot,\cdot)$ is as defined in Section~\ref{sec:gaussian sum ekf}. Assume $\mathbf{x}_{0}\sim\mathcal{N}(\hat{\mathbf{x}}_{0},\bm{\Sigma}_{0})$. Note that this assumption is also needed to initialize the EKF recursions. Using this in \eqref{eqn:RKHS-EKF joint density} along with \eqref{eqn:RKHS-EKF conditional density x} and \eqref{eqn:RKHS-EKF conditional density y}, we have
\par\noindent\small
\begin{align}
    &\log p(\mathbf{X}^{k},\mathbf{Y}^{k}|\Theta)=-\frac{1}{2}\log |\bm{\Sigma}_{0}|-\frac{1}{2}(\mathbf{x}_{0}-\hat{\mathbf{x}}_{0})^{T}\bm{\Sigma}_{0}^{-1}(\mathbf{x}_{0}-\hat{\mathbf{x}}_{0})\nonumber\\
    &\hspace{-0.3cm}+\sum_{j=1}^{k}\left(-\frac{1}{2}\log |\mathbf{Q}|-\frac{1}{2}(\mathbf{x}_{j}-\mathbf{A}\bm{\Phi}(\mathbf{x}_{j-1}))^{T}\mathbf{Q}^{-1}(\mathbf{x}_{j}-\mathbf{A}\bm{\Phi}(\mathbf{x}_{j-1}))\right)\nonumber\\
    &\hspace{-0.3cm}+\sum_{j=1}^{k}\left(-\frac{1}{2}\log |\mathbf{R}|-\frac{1}{2}(\mathbf{y}_{j}-\mathbf{B}\bm{\Phi}(\mathbf{x}_{j}))^{T}\mathbf{R}^{-1}(\mathbf{y}_{j}-\mathbf{B}\bm{\Phi}(\mathbf{x}_{j}))\right)+\alpha,\label{eqn:RKHS-EKF log all terms}
\end{align}
\normalsize
where $\alpha$ denotes the constant terms which do not affect the maximization.

A basic EM algorithm to estimate $\Theta$ based on the observations $\mathbf{Y}^{k}$ consists of following two steps that are iterated a fixed number of times or until convergence: \textit{a) E-step:} Given the estimate $\hat{\Theta}$ of the unknown parameters $\Theta$, we compute $Q(\Theta,\hat{\Theta})=\mathbb{E}[\log p(\mathbf{X}^{k},\mathbf{Y}^{k}|\Theta)|\mathbf{Y}^{k},\hat{\Theta}]$, and \textit{b) M-step:} The updated estimate $\hat{\Theta}^{(new)}=\textrm{arg max}_{\Theta} Q(\Theta,\hat{\Theta})$. Note that for online EM, at $k$-th time step, the current estimate $\hat{\Theta}$ is $\hat{\Theta}_{k-1}$ and the updated estimate $\hat{\Theta}^{(\textrm{new})}$ is $\hat{\Theta}_{k}$.

\textit{Approximate online EM for parameter estimates:} In our RKHS-EKF, we adopt an approximate online version for the EM algorithm such that all the available observations need not be processed at each time step. Further, as we will describe in the E-step, the required conditional expectations can be approximated using the estimates from EKF recursions itself. Hence, we first consider the M-step. Consider the $k$-th time step such that $k$ observations are available to us upto this instant.

\textit{Parameter estimates (M-step):} For simplicity, denote the conditional expectation operator $\mathbb{E}[\cdot|\mathbf{Y}^{k},\hat{\Theta}_{k-1}]$ given $k$ observations by $\mathbb{E}_{k}[\cdot]$. Now, from \eqref{eqn:RKHS-EKF log all terms}, maximizing $\mathbb{E}_{k}[\log p(\mathbf{X}^{k},\mathbf{Y}^{k}|\Theta)]$ with respect to parameter $\mathbf{A}$, we obtain the estimate $\hat{\mathbf{A}}_{k}$ of $\mathbf{A}$ given $k$ observations as
\par\noindent\small
\begin{align}
    \hat{\mathbf{A}}_{k}=\left(\sum_{j=1}^{k}\mathbb{E}_{k}[\mathbf{x}_{j}\bm{\Phi}(\mathbf{x}_{j-1})^{T}]\right)\left(\sum_{j=1}^{k}\mathbb{E}_{k}[\bm{\Phi}(\mathbf{x}_{j-1})\bm{\Phi}(\mathbf{x}_{j-1})^{T}]\right)^{-1}.\label{eqn:RKHS-EKF A estimate}
\end{align}
\normalsize

Note that these expectations are computed using (and hence, functions of) the observations $\mathbf{Y}^{k}$ and current parameter estimate $\hat{\Theta}_{k-1}$ as described later in the E-step. However, this computation requires all $k$ observations to be processed together at time $k$, and the complexity increases as $k$ increases. To obtain approximate online estimate at low computations, we define the sum $\mathbf{S}^{x\phi}_{k}=\sum_{j=1}^{k}\mathbb{E}_{k}[\mathbf{x}_{j}\bm{\Phi}(\mathbf{x}_{j-1})^{T}]$ and $\mathbf{S}^{\phi 1}_{k}=\sum_{j=1}^{k}\mathbb{E}_{k}[\bm{\Phi}(\mathbf{x}_{j-1})\bm{\Phi}(\mathbf{x}_{j-1})^{T}]$ such that $\hat{\mathbf{A}}_{k}=\mathbf{S}^{x\phi}_{k}(\mathbf{S}^{\phi 1}_{k})^{-1}$ from \eqref{eqn:RKHS-EKF A estimate}. We approximate these sums as
\par\noindent\small
\begin{align}
    &\mathbf{S}^{x\phi}_{k}\approx\mathbf{S}^{x\phi}_{k-1}+\mathbb{E}_{k}[\mathbf{x}_{k}\bm{\Phi}(\mathbf{x}_{k-1})^{T}],\label{eqn:RKHS-EKF sum x phi}\\
    &\mathbf{S}^{\phi 1}_{k}\approx\mathbf{S}^{\phi 1}_{k-1}+\mathbb{E}_{k}[\bm{\Phi}(\mathbf{x}_{k-1})\bm{\Phi}(\mathbf{x}_{k-1})^{T}].\label{eqn:RKHS-EKF sum phi 1}
\end{align}
\normalsize
\begin{remark}
This is an approximation because we have not considered an updated parameter estimate $\hat{\Theta}_{k-1}=\{\hat{\mathbf{A}}_{k-1},\hat{\mathbf{B}}_{k-1},\hat{\mathbf{Q}}_{k-1},\hat{\mathbf{R}}_{k-1}\}$ (obtained using observations upto time `$k-1$') for computing the conditional expectations in $\mathbf{S}^{x\phi}_{k-1}$ and $\mathbf{S}^{\phi 1}_{k-1}$. The current parameter estimates are used only in computing expectations $\mathbb{E}_{k}[\mathbf{x}_{k}\bm{\Phi}(\mathbf{x}_{k-1})^{T}]$ and $\mathbb{E}_{k}[\bm{\Phi}(\mathbf{x}_{k-1})\bm{\Phi}(\mathbf{x}_{k-1})^{T}]$.
\end{remark}
\begin{remark}
Typically, in an EM algorithm, the E and M-steps are iterated a number of times considering the most recent parameter estimates in each E-step. In our RKHS-EKF, because of these approximations, the current observation is considered only once to obtain the current state estimates and update the parameter estimates.
\end{remark}
Similarly, other approximate parameter updates (by approximating the summations as in \eqref{eqn:RKHS-EKF sum x phi} and \eqref{eqn:RKHS-EKF sum phi 1}) are obtained as
\par\noindent\small
\begin{align}
    &\hat{\mathbf{Q}}_{k}=\left(1-\frac{1}{k}\right)\hat{\mathbf{Q}}_{k-1}+\frac{1}{k}(\mathbb{E}_{k}[\mathbf{x}_{k}\mathbf{x}_{k}^{T}]-\hat{\mathbf{A}}_{k}\mathbb{E}_{k}[\bm{\Phi}(\mathbf{x}_{k-1})\mathbf{x}_{k}^{T}]\nonumber\\
    &-\mathbb{E}_{k}[\mathbf{x}_{k}\bm{\Phi}(\mathbf{x}_{k-1})^{T}]\hat{\mathbf{A}}_{k}^{T}+\hat{\mathbf{A}}_{k}\mathbb{E}_{k}[\bm{\Phi}(\mathbf{x}_{k-1})\bm{\Phi}(\mathbf{x}_{k-1})^{T}]\hat{\mathbf{A}}_{k}^{T}),\label{eqn:RKHS-EKF Q estimate}\\
    &\hat{\mathbf{B}}_{k}=\mathbf{S}^{y\phi}_{k}(\mathbf{S}^{\phi}_{k})^{-1},\label{eqn:RKHS-EKF B estimate}\\
    &\hat{\mathbf{R}}_{k}=\left(1-\frac{1}{k}\right)\hat{\mathbf{R}}_{k-1}+\frac{1}{k}(\mathbb{E}_{k}[\mathbf{y}_{k}\mathbf{y}_{k}^{T}]-\hat{\mathbf{B}}_{k}\mathbb{E}_{k}[\bm{\Phi}(\mathbf{x}_{k})\mathbf{y}_{k}^{T}]\nonumber\\
    &-\mathbb{E}_{k}[\mathbf{y}_{k}\bm{\Phi}(\mathbf{x}_{k})^{T}]\hat{\mathbf{B}}_{k}^{T}+\hat{\mathbf{B}}_{k}\mathbb{E}_{k}[\bm{\Phi}(\mathbf{x}_{k})\bm{\Phi}(\mathbf{x}_{k})^{T}]\hat{\mathbf{B}}_{k}^{T}),\label{eqn:RKHS-EKF R estimate}
\end{align}
\normalsize
where the sums $\mathbf{S}^{y\phi}_{k}=\sum_{j=1}^{k}\mathbb{E}_{k}[\mathbf{y}_{j}\bm{\Phi}(\mathbf{x}_{j})^{T}]$ and $\mathbf{S}^{\phi}_{k}=\sum_{j=1}^{k}\mathbb{E}_{k}[\bm{\Phi}(\mathbf{x}_{j})\bm{\Phi}(\mathbf{x}_{j})^{T}]$ are introduced to obtain approximate online estimates and are evaluated as $\mathbf{S}^{y\phi}_{k}=\mathbf{S}^{y\phi}_{k-1}+\mathbb{E}_{k}[\mathbf{y}_{k}\bm{\Phi}(\mathbf{x}_{k})^{T}]$ and $\mathbf{S}^{\phi}_{k}=\mathbf{S}^{\phi}_{k-1}+\mathbb{E}_{k}[\bm{\Phi}(\mathbf{x}_{k})\bm{\Phi}(\mathbf{x}_{k})^{T}]$. Further, using \eqref{eqn:RKHS-EKF observation approx}, $\mathbb{E}_{k}[\mathbf{y}_{k}\mathbf{y}_{k}^{T}]=\hat{\mathbf{B}}_{k}\mathbb{E}_{k}[\bm{\Phi}(\mathbf{x}_{k})\bm{\Phi}(\mathbf{x}_{k})^{T}]\hat{\mathbf{B}}_{k}^{T}+\hat{\mathbf{R}}_{k-1}$ and $\mathbb{E}_{k}[\mathbf{y}_{k}\bm{\Phi}(\mathbf{x}_{k})^{T}]=\hat{\mathbf{B}}_{k}\mathbb{E}_{k}[\bm{\Phi}(\mathbf{x}_{k})\bm{\Phi}(\mathbf{x}_{k})^{T}]$.

\textit{Expectation computations (E-step):} As mentioned earlier, we compute the required expectations in \eqref{eqn:RKHS-EKF sum x phi}-\eqref{eqn:RKHS-EKF R estimate} using the EKF estimates. A standard EKF recursion provides a Gaussian approximation of the conditional posterior distribution of the required states given all the observations available upto current time instant. However, 
the expectations here involve a non-linear transformation $\bm{\Phi}(\cdot)$. In EKF, first-order Taylor series expansion approximates these non-linear expectations. We also need the statistics of $\bm{\Phi}(\mathbf{x}_{k-1})$ given $\mathbf{Y}^{k}$. Hence, we consider an augmented state $\mathbf{z}_{k}=[\mathbf{x}_{k}^{T} \;\mathbf{x}_{k-1}^{T}]^{T}$ to obtain a smoothed estimate $\hat{\mathbf{x}}_{k-1|k}$ of the previous state $\mathbf{x}_{k-1}$ given $\mathbf{Y}^{k}$. Using these approximations, we formulate the RKHS-EKF to jointly compute estimates $\hat{\mathbf{x}}_{k}$ and $\hat{\Theta}_{k}$ 
as follows. The computation of the required expectations is detailed in the parameters update step (Section~\ref{subsec: recursion}-3).

\subsection{
Recursion}\label{subsec: recursion}
In terms of the augmented state $\mathbf{z}_{k}$, the RKHS-EKF system model is
\par\noindent\small
\begin{align}
    &\mathbf{z}_{k}=\widetilde{f}(\mathbf{z}_{k-1})+\widetilde{\mathbf{w}}_{k-1},\label{eqn:RKHS-EKF state transition aug}\\
    &\mathbf{y}_{k}=\widetilde{h}(\mathbf{z}_{k})+\mathbf{v}_{k},\label{eqn:RKHS-EKF observation aug}
\end{align}
\normalsize
where $\widetilde{f}(\mathbf{z}_{k-1})=[(\mathbf{A}\bm{\Phi}(\mathbf{x}_{k-1}))^{T}\;\mathbf{x}_{k-1}^{T}]^{T}$ and $\widetilde{h}(\mathbf{z}_{k})=\mathbf{B}\bm{\Phi}(\mathbf{x}_{k})$. The actual noise covariance matrix of $\widetilde{\mathbf{w}}_{k-1}=[\mathbf{w}_{k-1}^{T}\;\mathbf{0}_{1\times n}]^{T}$ is $\widetilde{\mathbf{Q}}=\begin{bmatrix}\mathbf{Q}&\mathbf{0}_{n\times n}\\\mathbf{0}_{n\times n}&\mathbf{0}_{n\times n}\end{bmatrix}$. At $k$-th time instant, we have from the previous recursion, an estimate $\hat{\mathbf{z}}_{k-1}=[\hat{\mathbf{x}}_{k-1|k-1}^{T}\; \hat{\mathbf{x}}_{k-2|k-1}^{T}]^{T}$ with the associated error covariance matrix $\bm{\Sigma}^{z}_{k-1}$, coefficient matrices estimates $\hat{\mathbf{A}}_{k-1}$ and $\hat{\mathbf{B}}_{k-1}$, and noise covariance matrices estimates $\hat{\mathbf{Q}}_{k-1}$ and $\hat{\mathbf{R}}_{k-1}$. We compute estimates $\hat{\mathbf{z}}_{k}$, $\hat{\mathbf{A}}_{k}$, $\hat{\mathbf{B}}_{k}$, $\hat{\mathbf{Q}}_{k}$ and $\hat{\mathbf{R}}_{k}$ based on the new observation $\mathbf{y}_{k}$ through the procedure summarized below.\\
\textit{1) Prediction:} Using $\mathbf{A}=\hat{\mathbf{A}}_{k-1}$ and $\widetilde{\mathbf{Q}}_{k-1}=\begin{bmatrix}\hat{\mathbf{Q}}_{k-1}&\mathbf{0}_{n\times n}\\\mathbf{0}_{n\times n}&\mathbf{0}_{n\times n}\end{bmatrix}$ in \eqref{eqn:RKHS-EKF state transition aug}, compute the predicted state and associated prediction error covariance matrix as
\par\noindent\small
\begin{align}
    \hat{\mathbf{z}}_{k|k-1}=\widetilde{f}(\hat{\mathbf{z}}_{k-1}),\;\;\;\bm{\Sigma}^{z}_{k|k-1}=\widetilde{\mathbf{F}}_{k-1}\bm{\Sigma}^{z}_{k-1}\widetilde{\mathbf{F}}_{k-1}^{T}+\widetilde{\mathbf{Q}}_{k-1},\label{eqn:RKHS-EKF prediction}
\end{align}
\normalsize
where $\widetilde{\mathbf{F}}_{k-1}\doteq\nabla\widetilde{f}(\mathbf{z})|_{\mathbf{z}=\hat{\mathbf{z}}_{k-1}}$.\\
\textit{2) Measurement update:} Using $\mathbf{B}=\hat{\mathbf{B}}_{k-1}$ and $\mathbf{R}=\hat{\mathbf{R}}_{k-1}$ in \eqref{eqn:RKHS-EKF observation aug}, compute
\par\noindent\small
\begin{align}
    &\mathbf{S}_{k}=\widetilde{\mathbf{H}}_{k}\bm{\Sigma}^{z}_{k|k-1}\widetilde{\mathbf{H}}_{k}^{T}+\mathbf{R},\label{eqn:RKHS-EKF measurement update Sk}\\
    &\hat{\mathbf{z}}_{k}=\hat{\mathbf{z}}_{k|k-1}+\bm{\Sigma}^{z}_{k|k-1}\widetilde{\mathbf{H}}_{k}^{T}\mathbf{S}_{k}^{-1}(\mathbf{y}_{k}-\widetilde{h}(\hat{\mathbf{z}}_{k|k-1})),\label{eqn:RKHS-EKF measurement update zk}\\
    &\bm{\Sigma}^{z}_{k}=\bm{\Sigma}^{z}_{k|k-1}-\bm{\Sigma}^{z}_{k|k-1}\widetilde{\mathbf{H}}_{k}^{T}\mathbf{S}_{k}^{-1}\widetilde{\mathbf{H}}_{k}\bm{\Sigma}^{z}_{k|k-1},\label{eqn:RKHS-EKF measurement update sigk}
\end{align}
\normalsize
where $\widetilde{\mathbf{H}}_{k}\doteq\nabla\widetilde{h}(\mathbf{z})|_{\mathbf{z}=\hat{\mathbf{z}}_{k|k-1}}$. Here, $\hat{\mathbf{z}}_{k}=[\hat{\mathbf{x}}_{k|k}^{T}\;\hat{\mathbf{x}}_{k-1|k}^{T}]^{T}$ where $\hat{\mathbf{x}}_{k|k}$ is the RKHS-EKF's estimate of state $\mathbf{x}_{k}$. The prediction and measurement update steps follow from the standard EKF recursions to estimate the augmented state $\mathbf{z}_{k}$ with the considered system model \eqref{eqn:RKHS-EKF state transition aug} and \eqref{eqn:RKHS-EKF observation aug}.\\
\textit{3) Parameters update:} Update the parameter estimates by approximating the required expectations in \eqref{eqn:RKHS-EKF sum x phi}-\eqref{eqn:RKHS-EKF R estimate}. Consider $\mathbb{E}_{k}[\mathbf{x}_{k}\bm{\Phi}(\mathbf{x}_{k-1})^{T}]$ from \eqref{eqn:RKHS-EKF sum x phi}. Based on the standard EKF, linearize $\bm{\Phi}(\cdot)$ as
\par\noindent\small
\begin{align*}
    \bm{\Phi}(\mathbf{x}_{k-1})\approx\bm{\Phi}(\hat{\mathbf{x}}_{k-1|k})+\nabla\bm{\Phi}(\hat{\mathbf{x}}_{k-1|k})(\mathbf{x}_{k-1}-\hat{\mathbf{x}}_{k-1|k}),
\end{align*}
\normalsize
where $\nabla\bm{\Phi}(\hat{\mathbf{x}}_{k-1|k})\doteq\nabla\bm{\Phi}(\mathbf{x})|_{\mathbf{x}=\hat{\mathbf{x}}_{k-1|k}}$. Also, similar to standard EKF, we assume negligible error in computation of conditional means, i.e., $\mathbb{E}_{k}[\mathbf{x}_{k}]\approx\hat{\mathbf{x}}_{k|k}$ and $\mathbb{E}_{k}[\mathbf{x}_{k-1}]\approx\hat{\mathbf{x}}_{k-1|k}$. These yield 
\par\noindent\small
\begin{align}
    &\mathbb{E}_{k}[\mathbf{x}_{k}\bm{\Phi}(\mathbf{x}_{k-1})^{T}]=\hat{\mathbf{x}}_{k|k}\bm{\Phi}(\hat{\mathbf{x}}_{k-1|k})^{T}\nonumber\\
    &\;\;\;+\textrm{Cov}(\mathbf{x}_{k}-\hat{\mathbf{x}}_{k|k},\mathbf{x}_{k-1}-\hat{\mathbf{x}}_{k-1|k})\nabla\bm{\Phi}(\hat{\mathbf{x}}_{k-1|k})^{T}.\label{eqn:Exphi}
\end{align}
\normalsize
By definition, $\textrm{Cov}(\mathbf{x}_{k}-\hat{\mathbf{x}}_{k|k},\mathbf{x}_{k-1}-\hat{\mathbf{x}}_{k-1|k})\approx[\bm{\Sigma}^{z}_{k}]_{(1:n,n+1:2n)}$, which was computed during the measurement update. Similarly, we have
\par\noindent\small
\begin{align}
    &\mathbb{E}_{k}[\bm{\Phi}(\mathbf{x}_{k-1})\bm{\Phi}(\mathbf{x}_{k-1})^{T}]=\bm{\Phi}(\hat{\mathbf{x}}_{k-1|k})\bm{\Phi}(\hat{\mathbf{x}}_{k-1|k})^{T}\nonumber\\
    &\;\;\;+\nabla\bm{\Phi}(\hat{\mathbf{x}}_{k-1|k})\textrm{Cov}(\mathbf{x}_{k-1}-\hat{\mathbf{x}}_{k-1|k})\nabla\bm{\Phi}(\hat{\mathbf{x}}_{k-1|k})^{T},\label{eqn:Ephi1}\\
    &\mathbb{E}_{k}[\bm{\Phi}(\mathbf{x}_{k})\bm{\Phi}(\mathbf{x}_{k})^{T}]=\bm{\Phi}(\hat{\mathbf{x}}_{k|k})\bm{\Phi}(\hat{\mathbf{x}}_{k|k})^{T}\nonumber\\
    &\;\;\;+\nabla\bm{\Phi}(\hat{\mathbf{x}}_{k|k})\textrm{Cov}(\mathbf{x}_{k}-\hat{\mathbf{x}}_{k|k})\nabla\bm{\Phi}(\hat{\mathbf{x}}_{k|k})^{T},\label{eqn:Ephi}\\
    &\mathbb{E}_{k}[\mathbf{x}_{k}\mathbf{x}_{k}^{T}]=\textrm{Cov}(\mathbf{x}_{k}-\hat{\mathbf{x}}_{k|k})+\hat{\mathbf{x}}_{k|k}\hat{\mathbf{x}}_{k|k}^{T},\label{eqn:Exx}
\end{align}
\normalsize
where $\textrm{Cov}(\mathbf{x}_{k-1}-\hat{\mathbf{x}}_{k-1|k})\approx[\bm{\Sigma}^{z}_{k}]_{(n+1:2n,n+1:2n)}$ and $\textrm{Cov}(\mathbf{x}_{k}-\hat{\mathbf{x}}_{k|k})\approx[\bm{\Sigma}^{z}_{k}]_{(1:n,1:n)}$. With these expectations, the updated parameter estimates $\hat{\Theta}_{k}$ can be computed using \eqref{eqn:RKHS-EKF sum x phi}-\eqref{eqn:RKHS-EKF R estimate}.

Finally, the dictionary $\{\widetilde{\mathbf{x}}_{l}\}_{1\leq l\leq L}$ is updated using the new estimate $\hat{\mathbf{x}}_{k|k}$ based on the sliding window or ALD criterion. The following Algorithms~\ref{alg:RKHS-EKF initialization} and \ref{alg:RKHS-EKF recursion} summarize the RKHS-EKF's initialization and recursions, respectively. Note that the dictionary size increases when $\hat{\mathbf{x}}_{k|k}$ is added to the dictionary based on the ALD criterion or when the current size L is less than the considered window length (initial transient phase) in the sliding window criterion. For simplicity, we initialize all mixing parameters (elements of $\hat{\mathbf{A}}_{0}$ and $\hat{\mathbf{B}}_{0}$) with one in Algorithm~\ref{alg:RKHS-EKF initialization}. However, the mixing parameters can be initialized with arbitrary values.

\begin{remark}\label{remark:RKHS-EKF complexity}
The computational complexity of RKHS-EKF depends on the state dimension `$n$' as well as the size of the dictionary `$L$'. In particular, the prediction and measurement update steps of RKHS-EKF follow from standard EKF recursion and hence, have a computational complexity of $\mathcal{O}(d^{3})$, where $d=2n$ is the dimension of augmented state $\mathbf{z}_{k}$. The parameters update step involves matrix multiplications of complexity $\mathcal{O}(n^{3})$ and matrix inversions of complexity $\mathcal{O}(L^{3})$. The updated parameters $\hat{\mathbf{A}}_{k}$ and $\hat{\mathbf{B}}_{k}$ require inversion of $L\times L$ matrices $\mathbf{S}^{\phi 1}_{k}$ and $\mathbf{S}^{\phi}_{k}$, respectively. Note that the dictionary size $L$ is a user-defined constant under the sliding window criterion. The conditions for a finite $L$ under the ALD criterion follow from \cite[Theorem~3.1]{engel2004kernel}.
\end{remark}
\begin{algorithm}
	\caption{RKHS-EKF initialization}
	\label{alg:RKHS-EKF initialization}
    \begin{algorithmic}[1]
    \Statex \textbf{Input:} $\hat{\mathbf{x}}_{0}$, $\bm{\Sigma}_{0}$
    \Statex \textbf{Output:} $\hat{\mathbf{z}}_{0}$, $\bm{\Sigma}^{z}_{0}$, $L$, $\{\widetilde{x}_{l}\}_{1\leq l\leq L}$, $\hat{\mathbf{A}}_{0}$, $\hat{\mathbf{B}}_{0}$, $\hat{\mathbf{Q}}_{0}$, $\hat{\mathbf{R}}_{0}$, $\mathbf{S}^{x\phi}_{0}$, $\mathbf{S}^{\phi 1}_{0}$, $\mathbf{S}^{y\phi}_{0}$ and $\mathbf{S}^{\phi}_{0}$
\State $\hat{\mathbf{z}}_{0}\gets [\hat{\mathbf{x}}_{0}^{T}\;\;\hat{\mathbf{x}}_{0}^{T}]^{T}$, and $\bm{\Sigma}^{z}_{0}\gets\begin{bmatrix}\bm{\Sigma}_{0}&\mathbf{0}_{n\times n}\\\mathbf{0}_{n\times n}&\bm{\Sigma}_{0}\end{bmatrix}$.
\State Set $L=1$ and $\widetilde{\mathbf{x}}_{1}\gets\hat{\mathbf{x}}_{0}$.
\State $\hat{\mathbf{A}}_{0}\gets\mathbf{1}_{n\times L}$ and $\hat{\mathbf{B}}_{0}\gets\mathbf{1}_{p\times L}$.
\State Initialize $\hat{\mathbf{Q}}_{0}$ and $\hat{\mathbf{R}}_{0}$ with some suitable p.d. noise covariance matrices.
\State Set $\mathbf{S}^{x\phi}_{0}=\mathbf{0}_{n\times L}$, $\mathbf{S}^{\phi 1}_{0}=\mathbf{0}_{L\times L}$, $\mathbf{S}^{y\phi}_{0}=\mathbf{0}_{p\times L}$ and $\mathbf{S}^{\phi}_{0}=\mathbf{0}_{L\times L}$.

\Statex \Return $\hat{\mathbf{z}}_{0}$, $\bm{\Sigma}^{z}_{0}$, $L$, $\{\widetilde{x}_{l}\}_{1\leq l\leq L}$, $\hat{\mathbf{A}}_{0}$, $\hat{\mathbf{B}}_{0}$, $\hat{\mathbf{Q}}_{0}$, $\hat{\mathbf{R}}_{0}$, $\mathbf{S}^{x\phi}_{0}$, $\mathbf{S}^{\phi 1}_{0}$, $\mathbf{S}^{y\phi}_{0}$ and $\mathbf{S}^{\phi}_{0}$.

    \end{algorithmic}
\end{algorithm}

\begin{algorithm}
	\caption{RKHS-EKF recursion 
	}
	\label{alg:RKHS-EKF recursion}
    \begin{algorithmic}[1]
    \Statex \textbf{Input:} $\hat{\mathbf{z}}_{k-1}$, $\bm{\Sigma}^{z}_{k-1}$, $\hat{\mathbf{A}}_{k-1}$, $\hat{\mathbf{B}}_{k-1}$, $\hat{\mathbf{Q}}_{k-1}$, $\hat{\mathbf{R}}_{k-1}$, $\mathbf{S}^{x\phi}_{k-1}$, $\mathbf{S}^{\phi 1}_{k-1}$, $\mathbf{S}^{y\phi}_{k-1}$, $\mathbf{S}^{\phi}_{k-1}$, and $\mathbf{y}_{k}$
    \Statex \textbf{Output:} $\hat{\mathbf{x}}_{k|k}$, $\hat{\mathbf{z}}_{k}$, $\bm{\Sigma}^{z}_{k}$, $\hat{\mathbf{A}}_{k}$, $\hat{\mathbf{B}}_{k}$, $\hat{\mathbf{Q}}_{k}$, $\hat{\mathbf{R}}_{k}$, $\mathbf{S}^{x\phi}_{k}$, $\mathbf{S}^{\phi 1}_{k}$, $\mathbf{S}^{y\phi}_{k}$, and $\mathbf{S}^{\phi}_{k}$
\State Compute $\hat{\mathbf{z}}_{k|k-1}$ and $\bm{\Sigma}^{z}_{k|k-1}$ using \eqref{eqn:RKHS-EKF prediction}.
\State Compute $\hat{\mathbf{z}}_{k}$ and $\bm{\Sigma}^{z}_{k}$ using \eqref{eqn:RKHS-EKF measurement update Sk}-\eqref{eqn:RKHS-EKF measurement update sigk}.
\State $\hat{\mathbf{x}}_{k|k}\gets[\hat{\mathbf{z}}_{k}]_{1:n}$.
\State $\textrm{Cov}(\mathbf{x}_{k}-\hat{\mathbf{x}}_{k|k},\mathbf{x}_{k-1}-\hat{\mathbf{x}}_{k-1|k})\gets[\bm{\Sigma}^{z}_{k}]_{(1:n,n+1:2n)}$, $\textrm{Cov}(\mathbf{x}_{k-1}-\hat{\mathbf{x}}_{k-1|k})\gets[\bm{\Sigma}^{z}_{k}]_{(n+1:2n,n+1:2n)}$ and $\textrm{Cov}(\mathbf{x}_{k}-\hat{\mathbf{x}}_{k|k})\gets[\bm{\Sigma}^{z}_{k}]_{(1:n,1:n)}$.
\State Compute $\mathbb{E}_{k}[\mathbf{x}_{k}\bm{\Phi}(\mathbf{x}_{k-1})^{T}]$, $\mathbb{E}_{k}[\bm{\Phi}(\mathbf{x}_{k-1})\bm{\Phi}(\mathbf{x}_{k-1})^{T}]$, $\mathbb{E}_{k}[\bm{\Phi}(\mathbf{x}_{k})\bm{\Phi}(\mathbf{x}_{k})^{T}]$ and $\mathbb{E}_{k}[\mathbf{x}_{k}\mathbf{x}_{k}^{T}]$ using \eqref{eqn:Exphi}-\eqref{eqn:Exx}.
\State Compute $\hat{\mathbf{A}}_{k}=\mathbf{S}^{x\phi}_{k}(\mathbf{S}^{\phi 1}_{k})^{-1}$, $\hat{\mathbf{B}}_{k}$, $\hat{\mathbf{Q}}_{k}$, and $\hat{\mathbf{R}}_{k}$ using \eqref{eqn:RKHS-EKF sum x phi}-\eqref{eqn:RKHS-EKF R estimate}.
\State Update dictionary $\{\widetilde{\mathbf{x}}_{l}\}_{1\leq l\leq L}$ using the state estimate $\hat{\mathbf{x}}_{k|k}$ based on the sliding window\cite{van2006sliding} or ALD\cite{engel2004kernel} criterion.
\If{dictionary size increases}
    \Statex Augment $\hat{\mathbf{A}}_{k}$, $\hat{\mathbf{B}}_{k}$, $\mathbf{S}^{x\phi}_{k}$, $\mathbf{S}^{\phi 1}_{k}$, $\mathbf{S}^{y\phi}_{k}$, and $\mathbf{S}^{\phi}_{k}$ with suitable initial values to take into account the updated dictionary size.
\EndIf
\Statex \Return $\hat{\mathbf{x}}_{k|k}$, $\hat{\mathbf{z}}_{k}$, $\bm{\Sigma}^{z}_{k}$, $\hat{\mathbf{A}}_{k}$, $\hat{\mathbf{B}}_{k}$, $\hat{\mathbf{Q}}_{k}$, $\hat{\mathbf{R}}_{k}$, $\mathbf{S}^{x\phi}_{k}$, $\mathbf{S}^{\phi 1}_{k}$, $\mathbf{S}^{y\phi}_{k}$, and $\mathbf{S}^{\phi}_{k}$.

    \end{algorithmic}
\end{algorithm}

\section{Numerical Experiments}\label{sec:simulations}
We considered two illustrative examples to demonstrate the performance of our proposed inverse filters. We compared their estimation errors with RCRLB \cite{tichavsky1998posterior} as a benchmark. The RCRLB $\mathbb{E}\left[(\mathbf{x}_{k}-\hat{\mathbf{x}}_{k})(\mathbf{x}_{k}-\hat{\mathbf{x}}_{k})^{T}\right]\succeq\mathbf{J}_{k}^{-1}$, where $\mathbf{J}_{k}$ is the Fisher information matrix, provides a lower bound on mean-squared error (MSE). For the non-linear system given by \eqref{eqn: state transition x} and \eqref{eqn: observation y}, the forward information matrix $\mathbf{J}_{k}$ recursions reduces to \cite{xiong2006performance_ukf}
\par\noindent\small
\begin{align}
    &\mathbf{J}_{k+1}=\mathbf{Q}_{k}^{-1}\nonumber\\
    &\;\;+\mathbf{H}_{k+1}^{T}\mathbf{R}_{k+1}^{-1}\mathbf{H}_{k+1}-\mathbf{Q}_{k}^{-1}\mathbf{F}_{k}(\mathbf{J}_{k}+\mathbf{F}_{k}^{T}\mathbf{Q}_{k}^{-1}\mathbf{F}_{k})^{-1}\mathbf{F}_{k}^{T}\mathbf{Q}_{k}^{-1},\label{eqn: additive Jk recursions}
\end{align}
\normalsize
where $\mathbf{F}_{k}=\nabla_{\mathbf{x}}f(\mathbf{x})\vert_{\mathbf{x}=\mathbf{x}_{k}}$ and $\mathbf{H}_{k}=\nabla_{\mathbf{x}}h(\mathbf{x})\vert_{\mathbf{x}=\mathbf{x}_{k}}$. These recursions can be trivially extended to other system models considered in this paper as well as to compute the posterior information matrix $\overline{\mathbf{J}}_{k}$ for inverse filter's estimate $\doublehat{\mathbf{x}}_{k}$. We refer the reader to Section VI of the companion paper (Part I) \cite{singh2022inverse_part1} for further details on computation of RCRLB for forward and inverse filters' state estimates. Some recent studies on cognitive radar target tracking instead consider posterior CRLB \cite{bell2015cognitive} as a metric to tune tracking filters.

In the following, we also consider the proposed inverse filters' performance under incorrect forward filter assumption such that the misspecified CRB (MCRB)\cite{richmond2015parameter,fortunati2017performance} may provide further insights. However, the MCRB as it exists in the literature, cannot be directly applied to the mismatched forward and inverse filter case. Firstly, MCRB's existence can be guaranteed only under certain regularity conditions, like the non-singularity of a generalized Fisher information matrix and the existence of a unique pseudo-true parameter vector minimizing the Kullback-Leibler divergence (KLD) between the true and assumed densities\cite{richmond2015parameter}. These regularity conditions cannot be trivially verified for our inverse filtering problem because of the non-linear transformation of the Gaussian noise terms. Furthermore, discrete-time filtering problems require a recursively computed lower bound as the observations $\mathbf{Y}^{k}=\{\mathbf{y}_{1},\mathbf{y}_{2},\hdots,\mathbf{y}_{k}\}$ and the states $\mathbf{X}^{k}=\{\mathbf{x}_{0},\mathbf{x}_{1},\hdots,\mathbf{x}_{k}\}$ upto $k$-th time step cannot be processed jointly. To the best of our knowledge, such recursive bounds for the misspecified non-linear filtering problems have not been proposed so far. Hence, in our experiments, we only consider RCRLB computed assuming a perfectly specified model.

\subsection{Inverse filters for SOEKF and GS-EKF}\label{subsec:sim EKF SOEKF GS-EKF}
\begin{figure}
  \centering
  \includegraphics[width = 1.0\columnwidth]{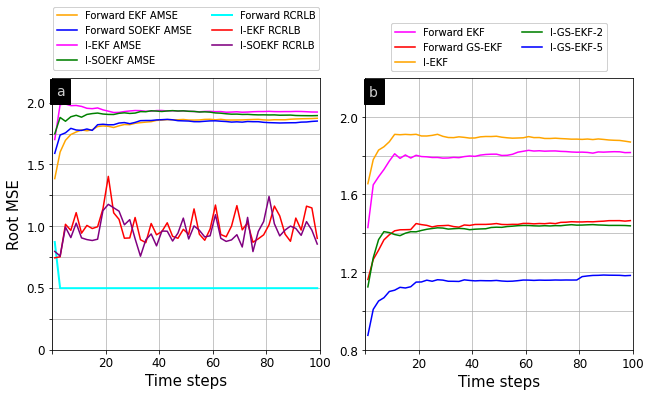}
  \caption{AMSE and RCRLB for forward and inverse filters: (a) SOEKF; (b) GS-EKF (I-GS-EKF-2 and I-GS-EKF-5, respectively, for $\overline{l}=2$ and $5$); compared to EKF and I-EKF, averaged over 500 runs.}
 \label{fig:EKF SOEKF GS-EKF}
\end{figure}
Recall from Section VI-B of the companion paper (Part I) \cite{singh2022inverse_part1}, the discrete-time non-linear system model of FM demodulator\cite[Sec. 8.2]{anderson2012optimal}. 
The state-transition and observation equations are
\par\noindent\small
\begin{align*}
&\mathbf{x}_{k+1}\doteq\begin{bmatrix}\lambda_{k+1}\\\theta_{k+1}\end{bmatrix}=\begin{bmatrix}\exp{(-T/\beta)}&0\\-\beta \exp{(-T/\beta)}-1&1\end{bmatrix}\begin{bmatrix}\lambda_{k}\\\theta_{k}\end{bmatrix}+\begin{bmatrix}1\\-\beta\end{bmatrix}w_{k},\\
&\mathbf{y}_{k}=\sqrt{2}\begin{bmatrix}\sin{\theta_{k}}\\\cos{\theta_{k}}\end{bmatrix}+\mathbf{v}_{k},\;\;
a_{k}=\hat{\lambda}_{k}^{2}+\epsilon_{k},
\end{align*}
\normalsize
with $w_{k}\sim\mathcal{N}(0,0.01)$, $\mathbf{v}_{k}\sim\mathcal{N}(\mathbf{0},\mathbf{I}_{2})$, $\epsilon_{k}\sim\mathcal{N}(0,5)$, $T=2\pi/16$ and $\beta=100$. Here, $\hat{\lambda}_{k}$ is the forward filter's estimate of $\lambda_{k}$.

The initial state $\mathbf{x}_{0}\doteq[\lambda_{0},\theta_{0}]^{T}$ and its estimates for all forward and inverse filters including mean estimates for GS-EKF were set randomly with $\lambda_{0}\sim\mathcal{N}(0,1)$ and $\theta_{0}\sim\mathcal{U}[-\pi,\pi]$. The initial covariances were set to $\bm{\Sigma}_{0}=10\mathbf{I}_{2}$ and $\overline{\bm{\Sigma}}_{0}=5\mathbf{I}_{2}$ for forward and inverse SOEKF. In the case of GS-EKF, we considered 5 Gaussians for the forward filter with the initial covariances and weights set to $10\mathbf{I}_{2}$ and $1/5$, respectively. For I-GS-EKF, an augmented state $\mathbf{z}_{k}$ of $\lbrace\overline{\mathbf{x}}_{i,k},c_{i,k}\rbrace_{1\leq i\leq 5}$ was considered resulting in a $15$-dimensional state vector with the initial weight estimates set to $1/5$ and the initial covariance estimates $\lbrace\overline{\bm{\Sigma}}_{j,0}\rbrace_{1\leq j\leq\overline{l}}$ as $5\mathbf{I}_{15}$. 
All other parameters of the system including the initial information matrix estimates were identical to those in Section VI-B in the companion paper (Part I) \cite{singh2022inverse_part1}. The Gaussian noise term $\mathbf{v}_{k+1}$ in I-GS-EKF's state transition \eqref{eqn: inverse GS-EKF weight transition} is transformed through a non-linear function $\gamma(\cdot,\cdot)$ such that \eqref{eqn: additive Jk recursions} is not applicable. The RCRLB in this case is derived using the general $\mathbf{J}_{k}$ recursions given by \cite[eq.~(21)]{tichavsky1998posterior}, which we omit here.

Fig. \ref{fig:EKF SOEKF GS-EKF} shows the time-averaged RMSE (AMSE) $=\sqrt{(\sum_{i=1}^{k}\|\mathbf{x}_{i}-\hat{\mathbf{x}}_{i}\|_{2}^{2})/nk}$ at $k$-th time step with $n$-dimensional actual state $\mathbf{x}_{i}$ and its estimate $\hat{\mathbf{x}}_{i}$, and RCRLB for both forward and inverse SOEKF and GS-EKF, averaged over $500$ runs. Fig. \ref{fig:EKF SOEKF GS-EKF} shows that the forward GS-EKF with $l=5$ performs better than both forward EKF and forward SOEKF. Considering second-order terms of the Taylor series expansion, in addition to the first-order terms, does not improve the estimation performance for this system as observed in Fig. \ref{fig:EKF SOEKF GS-EKF}a. When including second-order terms in the forward SOEKF yields smaller gains, then this is also reflected in the RCRLB of the inverse filters, i.e., I-EKF and I-SOEKF have similar lower bound values. 
Both I-EKF and I-SOEKF converge to the same steady-state estimation error values, which is also higher than that of the corresponding forward filters. Being suboptimal filters, the forward as well as inverse EKF and SOEKF do not achieve the RCRLB on the estimation error. However, the difference between AMSE and RCRLB for the inverse filters is less than that for the forward filters. We conclude that I-EKF and I-SOEKF are more efficient here. 

For GS-EKF in Fig. \ref{fig:EKF SOEKF GS-EKF}b, the estimation error of the inverse filter is same as that of the forward filter when $\overline{l}=2$ but improves significantly when $\overline{l}=5$. 
Note that this improvement in performance comes at the expense of increased computational complexity 
because the inverse filter estimates an augmented state of dimension `$l(n+1)$', which is larger than the forward filter's state dimension `$n$'. 

The I-SOEKF assume initial covariance $\bm{\Sigma}_{0}$ as $5\mathbf{I}_{2}$ (the true $\bm{\Sigma}_{0}$ of forward SOEKF is $10\mathbf{I}_{2}$) and a random initial state for these recursions. In spite of this difference in the initial estimates, I-SOEKF's error performance is comparable to that of the forward SOEKF. Interestingly, despite similar differences in the initial estimates, 
I-GS-EKF with $\overline{l}=5$ outperforms the forward GS-EKF. 
\begin{figure}
  \centering
  \includegraphics[width = 1.0\columnwidth]{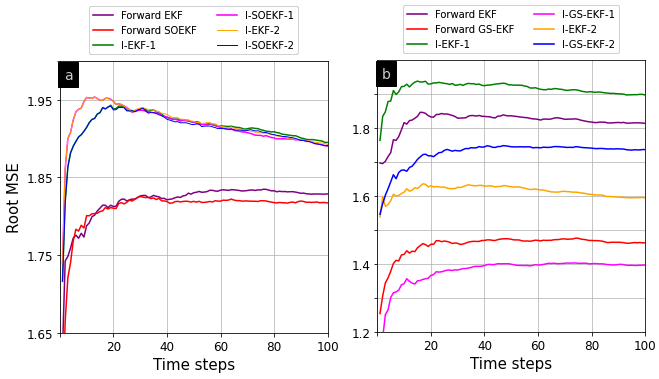}
  \caption{AMSE for forward and inverse filters: (a) I-EKF and I-SOEKF; and (b) I-EKF and I-GS-EKF. The assumptions of different inverse filters are detailed in Table~\ref{tbl:ekf mismatch describe}.}
 \label{fig:EKF mismatch}
\end{figure}
    \begin{table}
    \caption{Summary of forward-inverse filters 
    in Fig.~\ref{fig:EKF mismatch}}
    \label{tbl:ekf mismatch describe}
    \centering
    \begin{tabular}{p{0.4cm}p{2.4cm}p{2.7cm}p{1.5cm}}
    \hline\noalign{\smallskip}
    Fig. & True forward filter & Assumed forward filter & Inverse filter\\
    \noalign{\smallskip}
    \hline
    \noalign{\smallskip}
    \ref{fig:EKF mismatch}a & EKF & EKF & I-EKF-1\\
    \ref{fig:EKF mismatch}a & SOEKF & EKF & I-EKF-2\\
    \ref{fig:EKF mismatch}a & SOEKF & SOEKF & I-SOEKF-1\\
    \ref{fig:EKF mismatch}a & EKF & SOEKF & I-SOEKF-2\\
    \ref{fig:EKF mismatch}b & EKF & EKF & I-EKF-1\\
    \ref{fig:EKF mismatch}b & GS-EKF & EKF & I-EKF-2\\
    \ref{fig:EKF mismatch}b & GS-EKF & GS-EKF & I-GS-EKF-1\\
    \ref{fig:EKF mismatch}b & EKF & GS-EKF & I-GS-EKF-2\\
    \noalign{\smallskip}
    \hline\noalign{\smallskip}
    \end{tabular}
    \end{table}

In Fig.~\ref{fig:EKF mismatch}, we consider the case when the defender assumes a different forward filter from the actual adversary's filter. In Fig.~\ref{fig:EKF mismatch}b, the I-GS-EKF with $\overline{l}=5$ is considered. The adversary's true forward filter and that assumed by the defender for different inverse filters of Fig.~\ref{fig:EKF mismatch} are provided in Table~\ref{tbl:ekf mismatch describe}. The forward as well as inverse filters (with perfect information of forward filter) are also included for comparison.


From Fig.~\ref{fig:EKF mismatch}, we observe that the I-SOEKF and I-EKF have similar performance regardless of the true forward filter employed by the adversary because both the forward EKF and SOEKF also have same estimation accuracy. On the other hand, forward GS-EKF performs better than EKF for the considered system. Hence, I-EKF-2 based on observations from forward GS-EKF has lower estimation error than I-EKF-1 (with EKF as true forward filter) even though our assumption about forward filter is not true. However, I-EKF-2's estimation error is still more than I-GS-EKF-1 which knows the true forward filter. Furthermore, I-GS-EKF assuming forward GS-EKF even though the true forward filter is EKF (I-GS-EKF-2 case) also has lower estimation error than I-EKF-1, which uses perfect forward filter information. Hence, we conclude that assuming a more sophisticated forward GS-EKF and using I-GS-EKF provides higher estimation accuracy than I-EKF at the expense of computational efforts, regardless of the forward filter employed by the adversary.

\subsection{Inverse DEKF}\label{subsec:sim dithered EKF}
We consider the application of coordinate estimation of a stationary target from bearing observations taken by a moving sensor \cite{weiss1980improved}. The actual coordinates of the stationary target are $(X,Y)$ and that of the sensor at $k$-th time instant are $(p^{x}_{k},p^{y}_{k})$. The constant velocity of sensor is $s$. The forward and inverse EKF as well as DEKF were implemented in a modified coordinate basis with the state estimate $\mathbf{x}_{k}=[p^{x}_{k}/Y, s/Y, s, X/Y]^{T}$ and system model
\par\noindent\small
\begin{align*}
    &\mathbf{x}_{k+1}=\begin{bmatrix}1&\Delta t&0&0\\0&1&0&0\\0&0&1&0\\0&0&0&1\end{bmatrix}\mathbf{x}_{k}+\begin{bmatrix}0\\\Delta t/Y\\\Delta t\\0\end{bmatrix}w_{k},\\
    &y_{k}=\arctan([\mathbf{x}_{k}]_{4}-[\mathbf{x}_{k}]_{1})+v_{k},\;\;
    a_{k}=([\hat{\mathbf{x}}_{k}]_{4})^{2}+\epsilon_{k},
\end{align*}
\normalsize
where $w_{k}\sim\mathcal{N}(0,0.1^{2})$, $v_{k}\sim\mathcal{N}(0,2^{2})$, $\epsilon_{k}\sim\mathcal{N}(0,1.5^{2})$ and $\Delta t=20$ s. The initial covariance estimates were set to $\bm{\Sigma}_{0}=\textrm{diag}(4.44\times 10^{-7}, 0.5\times 10^{-6},1,0.1)$ and $\overline{\bm{\Sigma}}_{0}=\textrm{diag}(10^{-6},6\times 10^{-7},5,0.5)$, respectively, for the forward and inverse filters. 
The initial state estimate for inverse filters were $[0,0.002,200,2]^{T}$. The initial information matrices estimates $\mathbf{J}_{0}$ and $\overline{\mathbf{J}}_{0}$ were set to $\bm{\Sigma}_{0}$ and $\overline{\bm{\Sigma}}_{0}$, respectively. All other parameters of the system including the dither and the estimates were identical to those in \cite{weiss1980improved}.

With 200 time-steps, the modified observation function $h^{*}(\cdot)$ in the forward DEKF replaced $h(\cdot)$ up to 80 time-steps. 
Fig. \ref{fig:DEKF RKHS-EKF}a shows the absolute error and RCRLB, averaged over 400 runs, for estimation of $X/Y$ whose estimate at the $k$-th time instant are given by $[\hat{\mathbf{x}}_{k}]_{4}$ and $[\doublehat{\mathbf{x}}_{k}]_{4}$ of the forward and inverse filters. The RCRLB at $k$-th time instant is $\sqrt{\left[\mathbf{J}_{k}^{-1}\right]_{4,4}}$. 
The inverse filters' estimation errors were significantly lower than that of the forward filters. 
While I-DEKF-1 and I-DEKF-2 
differ in their transient performance, they converge to the same steady-state error as I-EKF. 
\begin{figure}
  \centering
  \includegraphics[width = 1.0\columnwidth]{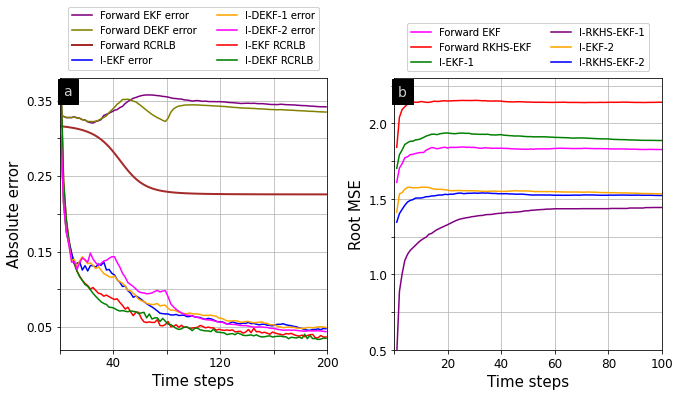}
  \caption{(a) Absolute error and RCRLB for forward and inverse EKF as well as DEKF. I-DEKF-1 is the I-DEKF without considering the modified dithered function in the inverse filter formulation while I-DEKF-2 considers the modified function. (b) AMSE for forward and inverse RKHS-EKF, compared to EKF and I-EKF, respectively. I-EKF-1 and I-EKF-2 assume the forward filter as EKF, but in case of I-EKF-2, the true forward filter is RKHS-EKF. The true forward filters are RKHS-EKF and EKF, respectively, for I-RKHS-EKF-1 and I-RKHS-EKF-2.}
 \label{fig:DEKF RKHS-EKF}
\end{figure}

\subsection{Inverse RKHS-EKF}\label{subsec:sim RKHS}
Consider the same FM-demodulator system as in Section~\ref{subsec:sim EKF SOEKF GS-EKF}. The defender employs RKHS-EKF to infer the adversary's state estimate while the forward filter is a simplified RKHS-EKF for known observation function. The initial state estimates for both filters were drawn at random similar to Section~\ref{subsec:sim EKF SOEKF GS-EKF}. For both forward and inverse RKHS-EKF, we considered the Gaussian kernel with kernel width $\sigma$ set to $\sigma^{2}=30$ and $50$, respectively. The initial covariance estimates ($\bm{\Sigma}^{z}_{0}$) were $10\mathbf{I}_{4}$ and $5\mathbf{I}_{4}$, respectively, for forward and inverse RKHS-EKF. To construct the forward RKHS-EKF's dictionary, we used the sliding window criterion\cite{van2006sliding} with window length $2$. On the other hand, the dictionary for inverse RKHS-EKF (I-RKHS-EKF) was constructed using ALD criterion \cite{engel2004kernel}. All the initial coefficient estimates (elements of $\hat{\mathbf{A}}_{0}$, and $\hat{\mathbf{B}}_{0}$) were set to $1$. The noise covariance estimates $\hat{\mathbf{Q}}_{0}$ and $\hat{\mathbf{R}}_{0}$ were $\textrm{diag}(1, 10)$ and $5$, respectively. All other system parameters were same as in Section~\ref{subsec:sim EKF SOEKF GS-EKF}.

Fig.~\ref{fig:DEKF RKHS-EKF}b shows the AMSE for forward and inverse RKHS-EKF, compared to EKF and I-EKF, respectively. We also include the case when the inverse filter assumes an incorrect forward filter. In particular, I-EKF assumes a forward EKF employed by the adversary. While I-EKF-1 considers the correct forward filter, the true forward filter is RKHS-EKF in case of I-EKF-2. On the other hand, I-RKHS-EKF does not make any such assumption. The true forward filters are RKHS-EKF and EKF for I-RKHS-EKF-1 and I-RKHS-EKF-2, respectively.

From Fig.~\ref{fig:DEKF RKHS-EKF}b, we observe that forward RKHS-EKF has higher estimation error than forward EKF because it does not have any prior information about the evolution of state $\mathbf{x}_{k}$. The I-EKF, with perfect knowledge of the system model, performs worse than the forward EKF. On the contrary, I-RKHS-EKF (I-RKHS-EKF-1 and I-RKHS-EKF-2) without any prior information, has smaller error than the corresponding forward filter. The improved performance of RKHS-based inverse filter is explained by the fact that it is a kernel function approximation and learns the unknown system dynamics completely based on its observations. On the other hand, I-EKF uses a first-order Taylor approximation for the non-linear functions. Here, the Gaussian kernel provides a better function approximation than linearization. Interestingly, with the true forward filter as RKHS-EKF, both I-RKHS-EKF-1 and I-EKF-2 have higher estimation accuracy than the forward filter, which is better than forward EKF and I-EKF-1 as well. This suggests that an inverse filter may have a lower estimation error even though the forward filter is inaccurate in its estimates.

\section{Summary}\label{sec:summary}
We developed the I-EKF theory of the companion paper (Part I) \cite{singh2022inverse_part1} to highly non-linear systems that employ inverses of SOEKF, GS-EKF, and DEKF. 
We showed that the stability of forward SOEKF is sufficient for I-SOEKF to be stable provided the system model satisfies certain additional assumptions. The inverse filters developed in the companion paper (Part I) \cite{singh2022inverse_part1} and those based on EKF variants are limited in the sense that they assume a known forward filter. The prior work \cite{krishnamurthy2019how} also assumes perfect system information for inverse filters. To this end, we proposed RKHS-EKF for such uncertain systems. The unknown functions are approximated using a kernel while the unknown parameters are learnt online using an approximate EM algorithm coupled with EKF recursions. Numerical examples demonstrate the estimation efficacy of the proposed filters with RCRLB as a performance measure. 
Our experiments with mismatched forward filters suggest that the inverse filter that assumes a similar advanced forward filter may provide better state estimates but at a high computational cost. While \cite{bell2015cognitive,sharaga2015optimal} suggest adapting a radar's waveform based on its observations, the developed inverse filters consider observing such a cognitive radar's adaptations and estimating its information for a counter-adversarial defender to guard against it.

\appendices

\section{Proof of Theorem~\ref{theorem: forward SOEKF stability}}
\label{App-thm-forward SOEKF stability}

\subsection{Preliminaries}

\begin{lemma}
\label{lemma: SOEKF stable bounds}
Under the assumptions of Theorem \ref{theorem: forward SOEKF stability}, the following bounds hold true for all $k\geq 0$.
\begin{enumerate}
    \item $\sum_{i=1}^{n}\sum_{j=1}^{n}\widetilde{\mathbf{a}}_{i}\widetilde{\mathbf{a}}_{j}^{T}\textrm{Tr}\left(\nabla^{2}\left[f(\hat{\mathbf{x}}_{k})\right]_{i}\bm{\Sigma}_{k}\nabla^{2}\left[f(\hat{\mathbf{x}}_{k})\right]_{j}\bm{\Sigma}_{k}\right)$ is a p.s.d. matrix and satisfies the upper bound
    \par\noindent\small
    \begin{align*}
        \sum_{i=1}^{n}\sum_{j=1}^{n}\widetilde{\mathbf{a}}_{i}\widetilde{\mathbf{a}}_{j}^{T}\textrm{Tr}\left(\nabla^{2}\left[f(\hat{\mathbf{x}}_{k})\right]_{i}\bm{\Sigma}_{k}\nabla^{2}\left[f(\hat{\mathbf{x}}_{k})\right]_{j}\bm{\Sigma}_{k}\right)\preceq\overline{a}^{2}\overline{\sigma}^{2}n^{2}\mathbf{I}.
    \end{align*}
    \normalsize
    \item $\sum_{i=1}^{p}\sum_{j=1}^{p}\mathbf{b}_{i}\mathbf{b}_{j}^{T}\textrm{Tr}\left(\nabla^{2}\left[h(\hat{\mathbf{x}}_{k})\right]_{i}\bm{\Sigma}_{k}\nabla^{2}\left[h(\hat{\mathbf{x}}_{k})\right]_{j}\bm{\Sigma}_{k}\right)$ is a p.s.d. matrix and satisfies the upper bound
    \par\noindent\small
    \begin{align*}
        \sum_{i=1}^{p}\sum_{j=1}^{p}\mathbf{b}_{i}\mathbf{b}_{j}^{T}\textrm{Tr}\left(\nabla^{2}\left[h(\hat{\mathbf{x}}_{k})\right]_{i}\bm{\Sigma}_{k}\nabla^{2}\left[h(\hat{\mathbf{x}}_{k})\right]_{j}\bm{\Sigma}_{k}\right)\preceq\overline{b}^{2}\overline{\sigma}^{2}np\mathbf{I}.
    \end{align*}
    \normalsize
    \item $\|\mathbf{M}_{k}\|\leq\beta$ with $\beta>0$.
\end{enumerate}
\end{lemma}
\begin{proof}
Using the bounds from the assumptions of Theorem \ref{theorem: forward SOEKF stability}, we have for all $i,j\in\lbrace1,2,\hdots,n\rbrace$
\par\noindent\small
\begin{align*}
\underline{a}^{2}\underline{\sigma}^{2}\mathbf{I}\preceq\nabla^{2}\left[f(\hat{\mathbf{x}}_{k})\right]_{i}\bm{\Sigma}_{k}\nabla^{2}\left[f(\hat{\mathbf{x}}_{k})\right]_{j}\bm{\Sigma}_{k}\preceq\overline{a}^{2}\overline{\sigma}^{2}\mathbf{I},
\end{align*}
which implies
\begin{align*}
\underline{a}^{2}\underline{\sigma}^{2}n\leq \textrm{Tr}\left(\nabla^{2}\left[f(\hat{\mathbf{x}}_{k})\right]_{i}\bm{\Sigma}_{k}\nabla^{2}\left[f(\hat{\mathbf{x}}_{k})\right]_{j}\bm{\Sigma}_{k}\right)\leq\overline{a}^{2}\overline{\sigma}^{2}n.
\end{align*}
\normalsize
Now, $\sum_{i=1}^{n}\sum_{j=1}^{n}\widetilde{\mathbf{a}}_{i}\widetilde{\mathbf{a}}_{j}^{T}$ is an $n\times n$ all-ones matrix with $n$ as one of its eigenvalue and all other ($n-1$) eigenvalues are zero. Hence, $\sum_{i=1}^{n}\sum_{j=1}^{n}\widetilde{\mathbf{a}}_{i}\widetilde{\mathbf{a}}_{j}^{T}\textrm{Tr}\left(\nabla^{2}\left[f(\hat{\mathbf{x}}_{k})\right]_{i}\bm{\Sigma}_{k}\nabla^{2}\left[f(\hat{\mathbf{x}}_{k})\right]_{j}\bm{\Sigma}_{k}\right)$ is a p.s.d. matrix and satisfies the first bound in the lemma. Similarly, the bound on $\sum_{i=1}^{p}\sum_{j=1}^{p}\mathbf{b}_{i}\mathbf{b}_{j}^{T}\textrm{Tr}\left(\nabla^{2}\left[h(\hat{\mathbf{x}}_{k})\right]_{i}\bm{\Sigma}_{k}\nabla^{2}\left[h(\hat{\mathbf{x}}_{k})\right]_{j}\bm{\Sigma}_{k}\right)$ can be derived. Further, the maximum singular value of $\sum_{i=1}^{n}\sum_{j=1}^{p}\widetilde{\mathbf{a}}_{i}\mathbf{b}_{j}^{T}$ is $\sqrt{np}$ and hence, using the bounds, we can show that $\|\mathbf{M}_{k}\|$ satisfies $\|\mathbf{M}_{k}\|\leq\beta$ with $\beta=\frac{1}{2}\overline{a}\overline{b}\overline{\sigma}^{2}n\sqrt{np}$.
\end{proof}

\begin{lemma}
\label{lemma: SOEKF stable alpha term}
Under the assumptions of Theorem \ref{theorem: forward SOEKF stability}, there exists a real number $\alpha$ with $0<\alpha<1$ such that
\par\noindent\small
\begin{align*}
    (\mathbf{F}_{k}-\mathbf{K}_{k}\mathbf{H}_{k})^{T}\bm{\Sigma}_{k+1}^{-1}(\mathbf{F}_{k}-\mathbf{K}_{k}\mathbf{H}_{k})\preceq(1-\alpha)\bm{\Sigma}_{k}^{-1}.
\end{align*}
\normalsize
\end{lemma}
\begin{proof}
Using \eqref{eqn: one step SOEKF sig tilde} and \eqref{eqn: one step SOEKF sig update}, we have
\par\noindent\small
\begin{align*}
    \bm{\Sigma}_{k+1}&=\mathbf{F}_{k}\bm{\Sigma}_{k}\mathbf{F}_{k}^{T}+\mathbf{Q}_{k}-\mathbf{K}_{k}\mathbf{S}_{k}\mathbf{K}_{k}^{T}\\
    &+\frac{1}{2}\sum_{i=1}^{n}\sum_{j=1}^{n}\widetilde{\mathbf{a}}_{i}\widetilde{\mathbf{a}}_{j}^{T}\textrm{Tr}\left(\nabla^{2}\left[f(\hat{\mathbf{x}}_{k})\right]_{i}\bm{\Sigma}_{k}\nabla^{2}\left[f(\hat{\mathbf{x}}_{k})\right]_{j}\bm{\Sigma}_{k}\right).
\end{align*}
\normalsize
Using the positive semi-definiteness of the last term (as proved in Lemma \ref{lemma: SOEKF stable bounds}) and substituting for $\mathbf{K}_{k}\mathbf{S}_{k}$ using \eqref{eqn: one step SOEKF Kk}, we have
\par\noindent\small
\begin{align*}
    \bm{\Sigma}_{k+1}\succeq\mathbf{F}_{k}\bm{\Sigma}_{k}\mathbf{F}_{k}^{T}+\mathbf{Q}_{k}-(\mathbf{F}_{k}\bm{\Sigma}_{k}\mathbf{H}_{k}^{T}+\mathbf{M}_{k})\mathbf{K}_{k}^{T}.
\end{align*}
\normalsize
Rearranging the terms,
\par\noindent\small
\begin{align*}
    \bm{\Sigma}_{k+1}\succeq&(\mathbf{F}_{k}-\mathbf{K}_{k}\mathbf{H}_{k})\bm{\Sigma}_{k}(\mathbf{F}_{k}-\mathbf{K}_{k}\mathbf{H}_{k})^{T}+\mathbf{Q}_{k}\\
    &+\mathbf{K}_{k}\mathbf{H}_{k}\bm{\Sigma}_{k}(\mathbf{F}_{k}-\mathbf{K}_{k}\mathbf{H}_{k})^{T}-\mathbf{M}_{k}\mathbf{K}_{k}^{T}.
\end{align*}
\normalsize
First, we consider the last two terms. Putting $\mathbf{M}_{k}=\mathbf{K}_{k}\mathbf{S}_{k}-\mathbf{F}_{k}\bm{\Sigma}_{k}\mathbf{H}_{k}^{T}$, the last terms can be expressed as
\par\noindent\small
\begin{align*}
    &\mathbf{K}_{k}\mathbf{H}_{k}\bm{\Sigma}_{k}(\mathbf{F}_{k}-\mathbf{K}_{k}\mathbf{H}_{k})^{T}-\mathbf{M}_{k}\mathbf{K}_{k}^{T}\\
    &=\mathbf{K}_{k}\mathbf{H}_{k}\bm{\Sigma}_{k}\mathbf{F}_{k}^{T}+\mathbf{F}_{k}\bm{\Sigma}_{k}\mathbf{H}_{k}^{T}\mathbf{K}_{k}^{T}-\mathbf{K}_{k}(\mathbf{S}_{k}+\mathbf{H}_{k}\bm{\Sigma}_{k}\mathbf{H}_{k}^{T})\mathbf{K}_{k}^{T}.
\end{align*}
\normalsize
Let $\mathbf{A}=\mathbf{K}_{k}\mathbf{H}_{k}\bm{\Sigma}_{k}\mathbf{F}_{k}^{T}$. Using the bounds of Theorem \ref{theorem: forward SOEKF stability} and Lemma \ref{lemma: SOEKF stable bounds}, we can show that $\|\mathbf{A}\|\leq\frac{\overline{f}\overline{\sigma}\overline{h}(\overline{f}\overline{\sigma}\overline{h}+\beta)}{\underline{r}}$. Also, $\mathbf{A}+\mathbf{A}^{T}$ is a symmetric `$n\times n$' matrix with
\par\noindent\small
\begin{align*}
    \|\mathbf{A}+\mathbf{A}^{T}\|\leq\frac{2\overline{f}\overline{\sigma}\overline{h}(\overline{f}\overline{\sigma}\overline{h}+\beta)}{\underline{r}}.
\end{align*}
\normalsize
This implies
\par\noindent\small
\begin{align*}
    -\frac{2\overline{f}\overline{\sigma}\overline{h}(\overline{f}\overline{\sigma}\overline{h}+\beta)}{\underline{r}}\mathbf{I}\preceq\mathbf{A}+\mathbf{A}^{T}\preceq\frac{2\overline{f}\overline{\sigma}\overline{h}(\overline{f}\overline{\sigma}\overline{h}+\beta)}{\underline{r}}\mathbf{I}.
\end{align*}
\normalsize
Using this and other bounds, we have
\par\noindent\small
\begin{align*}
    \mathbf{K}_{k}\mathbf{H}_{k}\bm{\Sigma}_{k}(\mathbf{F}_{k}-\mathbf{K}_{k}\mathbf{H}_{k})^{T}-\mathbf{M}_{k}\mathbf{K}_{k}^{T}\succeq-c\mathbf{I},
\end{align*}
\normalsize
where $c=\frac{2\overline{f}\overline{\sigma}\overline{h}(\overline{f}\overline{\sigma}\overline{h}+\beta)}{\underline{r}}+\left(2\overline{\sigma}\overline{h}^{2}+\delta+\frac{1}{2}\overline{b}^{2}\overline{\sigma}^{2}np\right)\left(\frac{\overline{f}\overline{\sigma}\overline{h}+\beta}{\underline{r}}\right)^{2}$ is a positive constant. Hence,
\par\noindent\small
\begin{align*}
    \bm{\Sigma}_{k+1}\succeq(\mathbf{F}_{k}-\mathbf{K}_{k}\mathbf{H}_{k})\bm{\Sigma}_{k}(\mathbf{F}_{k}-\mathbf{K}_{k}\mathbf{H}_{k})^{T}+\mathbf{Q}_{k}-c\mathbf{I}.
\end{align*}
\normalsize
Similar to the proof of \cite[Lemma 3.1]{reif1999stochastic}, we can prove $\mathbf{F}_{k}-\mathbf{K}_{k}\mathbf{H}_{k}$ to be invertible using matrix inversion lemma. Here,
\par\noindent\small
\begin{align*}
    \mathbf{F}_{k}^{-1}(\mathbf{F}_{k}-\mathbf{K}_{k}\mathbf{H}_{k})\bm{\Sigma}_{k}=\bm{\Sigma}_{k}-\bm{\Sigma}_{k}(\mathbf{H}_{k}^{T}+\bm{\Sigma}_{k}^{-1}\mathbf{F}_{k}^{-1}\mathbf{M}_{k})\mathbf{S}_{k}^{-1}\mathbf{H}_{k}\bm{\Sigma}_{k}.
\end{align*}
\normalsize
The R.H.S. is in the form of matrix inversion lemma:
\par\noindent\small
\begin{align*}
    (\mathbf{A}+\mathbf{U}\mathbf{C}\mathbf{V})^{-1}=\mathbf{A}^{-1}-\mathbf{A}^{-1}\mathbf{U}(\mathbf{C}^{-1}+\mathbf{V}\mathbf{A}^{-1}\mathbf{U})^{-1}\mathbf{V}\mathbf{A}^{-1},
\end{align*}
\normalsize
where $\mathbf{A}$ and $\mathbf{C}$ are invertible matrices. Hence, the matrix inversion lemma is applicable if $\mathbf{R}_{k}+\frac{1}{2}\sum_{i=1}^{p}\sum_{j=1}^{p}\mathbf{b}_{i}\mathbf{b}_{j}^{T}\textrm{Tr}\left(\nabla^{2}\left[h(\hat{\mathbf{x}}_{k})\right]_{i}\bm{\Sigma}_{k}\nabla^{2}\left[h(\hat{\mathbf{x}}_{k})\right]_{j}\bm{\Sigma}_{k}\right)-\mathbf{H}_{k}\mathbf{F}_{k}^{-1}\mathbf{M}_{k}$ is invertible (because $\mathbf{C}$ needs to be invertible). Since the difference of two matrices $\mathbf{X}-\mathbf{Y}$ is invertible if maximum singular value of $\mathbf{Y}$ is strictly less than the minimum singular value of $\mathbf{X}$, the required difference is invertible if $\|\mathbf{H}_{k}\mathbf{F}_{k}^{-1}\mathbf{M}_{k}\|<\underline{r}$ because $\mathbf{R}_{k}+\frac{1}{2}\sum_{i=1}^{p}\sum_{j=1}^{p}\mathbf{b}_{i}\mathbf{b}_{j}^{T}\textrm{Tr}\left(\nabla^{2}\left[h(\hat{\mathbf{x}}_{k})\right]_{i}\bm{\Sigma}_{k}\nabla^{2}\left[h(\hat{\mathbf{x}}_{k})\right]_{j}\bm{\Sigma}_{k}\right)\succeq\underline{r}\mathbf{I}$. But $\|\mathbf{H}_{k}\mathbf{F}_{k}^{-1}\mathbf{M}_{k}\|\leq\overline{h}\beta\|\mathbf{F}_{k}^{-1}\|$. Substituting the value of $\beta$ as derived in Lemma \ref{lemma: SOEKF stable bounds}, the sufficient condition for invertibility of the required matrix is
\par\noindent\small
\begin{align*}
    \|\mathbf{F}_{k}^{-1}\|<\frac{2\underline{r}}{\overline{h}\overline{a}\overline{b}\overline{\sigma}^{2}n\sqrt{np}}.
\end{align*}
\normalsize
The condition \eqref{eqn: SOEKF stable constraint on inverse norm} is sufficient for this to be satisfied. Note that this condition is a sufficient but not necessary condition for invertibility of $\mathbf{F}_{k}-\mathbf{K}_{k}\mathbf{H}_{k}$. With this condition, Lemma \ref{lemma: SOEKF stable alpha term} can be proved using similar approach as in \cite[Lemma 3.1]{reif1999stochastic} with
\par\noindent\small
\begin{align*}
    1-\alpha=\left(1+\frac{\underline{q}-c}{\overline{\sigma}(\overline{f}+(\overline{f}\overline{\sigma}\overline{h}^{2}+\beta\overline{h})/\underline{r})^{2}}\right)^{-1}.
\end{align*}
\normalsize
The condition \eqref{eqn: SOEKF stable constraint on q} is sufficient for $0<\alpha<1$. 

Substituting for constant $c$ and using $\underline{q}\leq\delta$ and $\underline{r}\leq\delta$, the inequality \eqref{eqn: SOEKF stable constraint on q} requires the following inequality to be satisfied
\par\noindent\small
\begin{align*}
    \frac{2\overline{f}\overline{\sigma}\overline{h}(\overline{f}\overline{\sigma}\overline{h}+\beta)}{\delta}+\left(2\overline{\sigma}\overline{h}^{2}+\delta+\frac{1}{2}\overline{b}^{2}\overline{\sigma}^{2}np\right)\left(\frac{\overline{f}\overline{\sigma}\overline{h}+\beta}{\delta}\right)^{2}<\delta.
\end{align*}
\normalsize
Rearranging the terms, we have
\par\noindent\small
\begin{align*}
    \delta^{3}-(\overline{f}\overline{\sigma}\overline{h}+\beta)(3\overline{f}\overline{\sigma}\overline{h}+\beta)\delta-(\overline{f}\overline{\sigma}\overline{h}+\beta)^{2}\left(2\overline{\sigma}\overline{h}^{2}+\frac{1}{2}\overline{b}^{2}\overline{\sigma}^{2}np\right)> 0.
\end{align*}
\normalsize
It can be observed that this inequality requires the noise bound $\delta$ to be large enough and other bounds on matrices like $\overline{f},\overline{h}$ etc. to be small.
\end{proof}

\begin{lemma}
\label{lemma: SOEKF stable r term}
Under the assumptions of Theorem \ref{theorem: forward SOEKF stability}, there exists positive real constants $\kappa_\textrm{nonl}$, $\epsilon'$ such that
\par\noindent\small
\begin{align*}
    \mathbf{r}_{k}^{T}\bm{\Sigma}_{k+1}^{-1}(2(\mathbf{F}_{k}-\mathbf{K}_{k}\mathbf{H}_{k})\mathbf{e}_{k}+\mathbf{r}_{k})\leq\kappa_\textrm{nonl}\|\mathbf{e}_{k}\|_{2}^{4},
\end{align*}
\normalsize
for $\|\mathbf{e}_{k}\|_{2}\leq\epsilon'$.
\end{lemma}
\begin{proof}
The proof follows from \cite[Lemma 3.2]{reif1999stochastic} with $\kappa_\textrm{nonl}=\frac{\kappa'}{\underline{\sigma}}\left(2\left(\overline{f}+\frac{\overline{f}\overline{\sigma}\overline{h}^{2}+\beta\overline{h}}{\underline{r}}\right)+\kappa'\epsilon'^{2}\right)$, where $\kappa'=\kappa_{\phi}+\kappa_{\chi}\frac{\overline{f}\overline{\sigma}\overline{h}+\beta}{\underline{r}}$ and $\epsilon'=\textrm{min}(\epsilon_{\phi},\epsilon_{\chi})$.
\end{proof}

\begin{lemma}
\label{lemma: SOEKF stable noise term}
Under the assumptions of Theorem \ref{theorem: forward SOEKF stability}, there exists positive real constant $\kappa_\textrm{noise}$ independent of $\delta$ such that
\par\noindent\small
\begin{align*}
\mathbb{E}[\mathbf{s}_{k}^{T}\bm{\Sigma}_{k+1}^{-1}\mathbf{s}_{k}]\leq\kappa_\textrm{noise}\delta.
\end{align*}
\normalsize
\end{lemma}
\begin{proof}
The proof follows from \cite[Lemma 3.3]{reif1999stochastic} with $\kappa_\textrm{noise}=\frac{n}{\underline{\sigma}}+\frac{\overline{f}^{2}\overline{h}^{2}\overline{\sigma}^{2}p}{\underline{\sigma}\underline{r}^{2}}$.
\end{proof}

\begin{lemma}
\label{lemma: SOEKF stable q term}
Under the assumptions of Theorem \ref{theorem: forward SOEKF stability}, there exist positive real constants $\kappa_\textrm{sec}$, $c_\textrm{sec}$ such that
\par\noindent\small
\begin{align*}
\mathbf{q}_{k}^{T}\bm{\Sigma}_{k+1}^{-1}(2(\mathbf{F}_{k}-\mathbf{K}_{k}\mathbf{H}_{k})\mathbf{e}_{k}+2\mathbf{r}_{k}+\mathbf{q}_{k})\leq\kappa_\textrm{sec}\|\mathbf{e}_{k}\|_{2}^{3}+c_\textrm{sec},
\end{align*}
\normalsize
for $\|\mathbf{e}_{k}\|_{2}\leq\epsilon'$ where $\epsilon'$ is same as in Lemma \ref{lemma: SOEKF stable r term}.
\end{lemma}
\begin{proof}
Using the upper bound on $\nabla^{2}\left[f(\hat{\mathbf{x}}_{k})\right]_{i}$ from the assumptions of Theorem \ref{theorem: forward SOEKF stability}, we have
\par\noindent\small
\begin{align*}
&\mathbf{e}_{k}^{T}\nabla^{2}\left[f(\hat{\mathbf{x}}_{k})\right]_{i}\mathbf{e}_{k}\leq\overline{a}\|\mathbf{e}_{k}\|_{2}^{2},
\end{align*}
which implies
\begin{align*}
&\|\sum_{i=1}^{n}\widetilde{\mathbf{a}}_{i}\mathbf{e}_{k}^{T}\nabla^{2}\left[f(\hat{\mathbf{x}}_{k})\right]_{i}\mathbf{e}_{k}\|_{2}\leq\overline{a}n\|\mathbf{e}_{k}\|_{2}^{2}.
\end{align*}
\normalsize
Similarly, using other bounds, it is trivial to show that for all $k\geq 0$, the following bounds are satisfied
\par\noindent\small
\begin{align*}
&\|\sum_{i=1}^{n}\widetilde{\mathbf{a}}_{i}\textrm{Tr}\left(\nabla^{2}\left[f(\hat{\mathbf{x}}_{k})\right]_{i}\bm{\Sigma}_{k}\right)\|_{2}\leq\overline{a}\overline{\sigma}n^{2},\\
&\|\mathbf{K}_{k}\sum_{i=1}^{p}\mathbf{b}_{i}\mathbf{e}_{k}^{T}\nabla^{2}\left[h(\hat{\mathbf{x}}_{k})\right]_{i}\mathbf{e}_{k}\|_{2}\leq\frac{\overline{b}p(\overline{f}\overline{\sigma}\overline{h}+\beta)}{\underline{r}}\|\mathbf{e}_{k}\|_{2}^{2},\\
&\|\mathbf{K}_{k}\sum_{i=1}^{p}\mathbf{b}_{i}\textrm{Tr}\left(\nabla^{2}\left[h(\hat{\mathbf{x}}_{k})\right]_{i}\bm{\Sigma}_{k}\right)\|_{2}\leq\frac{\overline{b}\overline{\sigma}np(\overline{f}\overline{\sigma}\overline{h}+\beta)}{\underline{r}}.
\end{align*}
\normalsize
Using these bounds, we have
\par\noindent\small
\begin{align*}
    \|\mathbf{q}_{k}\|_{2}\leq\kappa_{q}\|\mathbf{e}_{k}\|_{2}^{2}+c_{q},
\end{align*}
\normalsize
where $\kappa_{q}=\frac{1}{2}\left(\overline{a}n+\frac{\overline{b}p(\overline{f}\overline{\sigma}\overline{h}+\beta)}{\underline{r}}\right)$ and $c_{q}=\frac{1}{2}\left(\overline{a}\overline{\sigma}n^{2}+\frac{\overline{b}\overline{\sigma}np(\overline{f}\overline{\sigma}\overline{h}+\beta)}{\underline{r}}\right)$.

Using the bounds on $\|\mathbf{q}_{k}\|_{2}$ and $\|\mathbf{r}_{k}\|_{2}$ as used in \cite[Lemma 3.2]{reif1999stochastic}, the lemma can be proved with
\footnotesize
\begin{align*}
&\kappa_\textrm{sec}=\frac{\kappa_{q}}{\underline{\sigma}}\left(2\left(\overline{f}+\frac{\overline{f}\overline{\sigma}\overline{h}^{2}+\beta\overline{h}}{\underline{r}}\right)+2\kappa'\epsilon'^{2}+\kappa_{q}\epsilon'\right),\text{and}\\
&c_\textrm{sec}=\frac{c_{q}^{2}}{\underline{\sigma}}+\frac{\kappa_{q}c_{q}\epsilon'^{2}}{\underline{\sigma}}+\frac{c_{q}\epsilon'}{\underline{\sigma}}\left(2\left(\overline{f}+\frac{\overline{f}\overline{\sigma}\overline{h}^{2}+\beta\overline{h}}{\underline{r}}\right)+2\kappa'\epsilon'^{2}+\kappa_{q}\epsilon'\right).
\end{align*}
\normalsize
\end{proof}

\subsection{Proof of the theorem}
Consider $V_{k}(\mathbf{e}_{k})=\mathbf{e}_{k}^{T}\bm{\Sigma}_{k}^{-1}\mathbf{e}_{k}$. Using the bounds on $\bm{\Sigma}_{k}$ from assumptions of Theorem \ref{theorem: forward SOEKF stability}, we have
\par\noindent\small
\begin{align*}
    \frac{1}{\overline{\sigma}}\|\mathbf{e}_{k}\|_{2}^{2}\leq V_{k}(\mathbf{e}_{k})\leq\frac{1}{\underline{\sigma}}\|\mathbf{e}_{k}\|_{2}^{2}.
\end{align*}
\normalsize
Also, substituting for $\mathbf{e}_{k+1}$ using \eqref{eqn: forward SOEKF error}, we have
\par\noindent\small
\begin{align*}
V_{k+1}(\mathbf{e}_{k+1})&=\mathbf{e}_{k}^{T}(\mathbf{F}_{k}-\mathbf{K}_{k}\mathbf{H}_{k})^{T}\bm{\Sigma}_{k+1}^{-1}(\mathbf{F}_{k}-\mathbf{K}_{k}\mathbf{H}_{k})\mathbf{e}_{k}\\
&+\mathbf{r}_{k}^{T}\bm{\Sigma}_{k+1}^{-1}(2(\mathbf{F}_{k}-\mathbf{K}_{k}\mathbf{H}_{k})\mathbf{e}_{k}+\mathbf{r}_{k})\\
&+\mathbf{q}_{k}^{T}\bm{\Sigma}_{k+1}^{-1}(2(\mathbf{F}_{k}-\mathbf{K}_{k}\mathbf{H}_{k})\mathbf{e}_{k}+2\mathbf{r}_{k}+\mathbf{q}_{k})\\
&+\mathbf{s}_{k}^{T}\bm{\Sigma}_{k+1}^{-1}\mathbf{s}_{k}+2\mathbf{s}_{k}^{T}\bm{\Sigma}_{k+1}^{-1}((\mathbf{F}_{k}-\mathbf{K}_{k}\mathbf{H}_{k})\mathbf{e}_{k}+\mathbf{r}_{k}+\mathbf{q}_{k}).
\end{align*}
\normalsize
Using Lemmata \ref{lemma: SOEKF stable alpha term}, \ref{lemma: SOEKF stable r term}, and \ref{lemma: SOEKF stable q term}, we have for $\|\mathbf{e}_{k}\|_{2}\leq\epsilon'$
\par\noindent\small
\begin{align*}
V_{k+1}(\mathbf{e}_{k+1})\leq&(1-\alpha)V_{k}(\mathbf{e}_{k})+\kappa_\textrm{nonl}\|\mathbf{e}_{k}\|_{2}^{4}+\kappa_\textrm{sec}\|\mathbf{e}_{k}\|_{2}^{3}+c_\textrm{sec}\\
&+\mathbf{s}_{k}^{T}\bm{\Sigma}_{k+1}^{-1}\mathbf{s}_{k}+2\mathbf{s}_{k}^{T}\bm{\Sigma}_{k+1}^{-1}((\mathbf{F}_{k}-\mathbf{K}_{k}\mathbf{H}_{k})\mathbf{e}_{k}+\mathbf{r}_{k}+\mathbf{q}_{k}).
\end{align*}
\normalsize
The last term $\mathbf{s}_{k}^{T}\bm{\Sigma}_{k+1}^{-1}((\mathbf{F}_{k}-\mathbf{K}_{k}\mathbf{H}_{k})\mathbf{e}_{k}+\mathbf{r}_{k}+\mathbf{q}_{k})$ vanishes on taking expectation conditioned on $\mathbf{e}_{k}$ and hence, for $\|\mathbf{e}_{k}\|_{2}\leq\epsilon'$,
\par\noindent\small
\begin{align*}
\mathbb{E}[]V_{k+1}(\mathbf{e}_{k+1})\vert\mathbf{e}_{k}]\leq&(1-\alpha)V_{k}(\mathbf{e}_{k})+\kappa_\textrm{nonl}\|\mathbf{e}_{k}\|_{2}^{4}+\kappa_\textrm{sec}\|\mathbf{e}_{k}\|_{2}^{3}\\
&+c_\textrm{sec}+\kappa_{\textrm{noise}}\delta,
\end{align*}
\normalsize
where the bound of Lemma \ref{lemma: SOEKF stable noise term} is applied. But, for $\|\mathbf{e}_{k}\|_{2}\leq\epsilon'$,
\par\noindent\small
\begin{align*}
\kappa_\textrm{nonl}\|\mathbf{e}_{k}\|_{2}^{4}+\kappa_\textrm{sec}\|\mathbf{e}_{k}\|_{2}^{3}\leq(\kappa_\textrm{nonl}\epsilon'+\kappa_\textrm{sec})\|\mathbf{e}_{k}\|_{2}\|\mathbf{e}_{k}\|_{2}^{2}.
\end{align*}
\normalsize
Choosing $\epsilon=\textrm{min}\left(\epsilon',\frac{\alpha}{2\overline{\sigma}(\kappa_\textrm{nonl}\epsilon'+\kappa_\textrm{sec})}\right)$, we have for $\|\mathbf{e}_{k}\|_{2}\leq\epsilon$,
\par\noindent\small
\begin{align*}
\kappa_\textrm{nonl}\|\mathbf{e}_{k}\|_{2}^{4}+\kappa_\textrm{sec}\|\mathbf{e}_{k}\|_{2}^{3}\leq\frac{\alpha}{2}V_{k}(\mathbf{e}_{k}),
\end{align*}
which implies
\begin{align*}
\mathbb{E}[V_{k+1}(\mathbf{e}_{k+1})\vert\mathbf{e}_{k}]-V_{k}(\mathbf{e}_{k})\leq-\frac{\alpha}{2}V_{k}(\mathbf{e}_{k})+c_\textrm{sec}+\kappa_\textrm{noise}\delta.
\end{align*}
\normalsize
Hence, Lemma 1 of the companion paper (Part I) \cite{singh2022inverse_part1} applies here. However, to have negative mean drift, we require $\delta$ to be small enough such that there exists some $\widetilde{\epsilon}<\epsilon$ to satisfy \eqref{eqn: SOEKF stable constraint on delta}. This condition ensures that  for $\widetilde{\epsilon}\leq\|\mathbf{e}_{k}\|_{2}\leq\epsilon$, $\mathbb{E}[ V_{k+1}(\mathbf{e}_{k+1})\vert\mathbf{e}_{k}]-V_{k}(\mathbf{e}_{k})\leq 0$ is fulfilled and the estimation error $\mathbf{e}_{k}$ remains exponentially bounded in mean-squared sense if the error is within $\epsilon$ bound.

\section{Proof of Theorem~\ref{theorem: inverse SOEKF stability}}
\label{App-thm-inverse SOEKF stability}
From the state transition function of I-SOEKF \eqref{eqn: one step I-SOEKF state transition}, we have
\par\noindent\small
\begin{align*}
\nabla^{2}\left[\overline{f}_{k}(\doublehat{\mathbf{x}}_{k})\right]_{i}=\nabla^{2}\left[f(\doublehat{\mathbf{x}}_{k})\right]_{i}-\sum_{j=1}^{p}[\mathbf{K}_{k}]_{i,j}\nabla^{2}\left[h(\doublehat{\mathbf{x}}_{k})\right]_{j}.
\end{align*}
\normalsize
Using the upper bounds on Hessian matrices from the assumptions of Theorem \ref{theorem: forward SOEKF stability}, we have
\par\noindent\small
\begin{align*}
\nabla^{2}\left[\overline{f}_{k}(\doublehat{\mathbf{x}}_{k})\right]_{i}&\preceq\left(\overline{a}-\underline{b}\sum_{j=1}^{p}[\mathbf{K}_{k}]_{i,j}\right)\mathbf{I}\preceq\left(\overline{a}+|\underline{b}|\sum_{j=1}^{p}|[\mathbf{K}_{k}]_{i,j}|\right)\mathbf{I}.
\end{align*}
\normalsize
But, the row sum $\sum_{i=1}^{p}|[\mathbf{K}_{k}]_{i,j}|\leq\|\mathbf{K}_{k}\|_{\infty}$. Further, using equivalence of norms i.e. $\|\mathbf{K}_{k}\|_{\infty}\leq\sqrt{p}\|\mathbf{K}_{k}\|$, we have
\par\noindent\small
\begin{align*}
\nabla^{2}\left[\overline{f}_{k}(\doublehat{\mathbf{x}}_{k})\right]_{i}\preceq\left(\overline{a}+|\underline{b}|\sqrt{p}\left(\frac{\overline{f}\overline{\sigma}\overline{h}+\beta}{\underline{r}}\right)\right)\mathbf{I}.
\end{align*}
\normalsize
Similarly,
\par\noindent\small
\begin{align*}
\nabla^{2}\left[\overline{f}_{k}(\doublehat{\mathbf{x}}_{k})\right]_{i}&\succeq\left(\underline{a}-\overline{b}\sum_{j=1}^{p}[\mathbf{K}_{k}]_{i,j}\right)\mathbf{I}\succeq\left(\underline{a}-\frac{\overline{b}\sqrt{p}(\overline{f}\overline{\sigma}\overline{h}+\beta)}{\underline{r}}\right)\mathbf{I}.
\end{align*}
\normalsize
Hence, with $\underline{d}=\underline{a}-(\overline{b}\sqrt{p}(\overline{f}\overline{\sigma}\overline{h}+\beta)/\underline{r})$ and $\overline{d}=\overline{a}+(|\underline{b}|\sqrt{p}(\overline{f}\overline{\sigma}\overline{h}+\beta)/\underline{r})$, we have $\underline{d}\mathbf{I}\preceq\nabla^{2}\left[\overline{f}_{k}(\doublehat{x}_{k})\right]_{i}\preceq\overline{d}\mathbf{I}$ for all $i\in\lbrace 1,2,\hdots,n\rbrace$.

The proof of the remaining conditions of Theorem \ref{theorem: forward SOEKF stability} for I-SOEKF dynamics follows from the proof of Theorem 5 in the companion paper (Part I) \cite{singh2022inverse_part1}.
\bibliographystyle{IEEEtran}
\bibliography{main}
\end{document}